\documentclass[11pt,reqno,tbtags,a4paper]{amsart}
\usepackage{amssymb}
\usepackage{url}
\usepackage[square,numbers]{natbib}
\bibpunct[, ]{[}{]}{;}{n}{,}{,}
{\makeatletter
\gdef\section{\@startsection{section}{1}%
  \z@{.7\linespacing\@plus\linespacing}{.5\linespacing}%
  {\normalfont\bfseries\centering}}
}

\title
{Renewal theory for asymmetric $U$-statistics}

\date{14 April, 2018}

\author{Svante Janson}
\thanks{Partly supported by the Knut and Alice Wallenberg Foundation}
\address{Department of Mathematics, Uppsala University, PO Box 480,
SE-751~06 Uppsala, Sweden}
\email{svante.janson@math.uu.se}
\urladdr{http://www.math.uu.se/svante-janson}


\numberwithin{equation}{section}

\renewcommand\le{\leqslant}
\renewcommand\ge{\geqslant}

\allowdisplaybreaks

\theoremstyle{plain}
\newtheorem{theorem}{Theorem}[section]
\newtheorem{lemma}[theorem]{Lemma}
\newtheorem{proposition}[theorem]{Proposition}
\newtheorem{corollary}[theorem]{Corollary}

\theoremstyle{definition}
\newtheorem{example}[theorem]{Example}
\newtheorem{remark}[theorem]{Remark}

\theoremstyle{remark}

\newenvironment{romenumerate}[1][-10pt]{
\addtolength{\leftmargini}{#1}\begin{enumerate}
 \renewcommand{\labelenumi}{\textup{(\roman{enumi})}}%
 }{\end{enumerate}}

\newcounter{oldenumi}
{\setcounter{oldenumi}{\value{enumi}}
\begin{romenumerate} \setcounter{enumi}{\value{oldenumi}}}
{\end{romenumerate}}

\newcounter{thmenumerate}
\newenvironment{thmenumerate}
{\setcounter{thmenumerate}{0}%
 \def\item{\par
 \refstepcounter{thmenumerate}\textup{(\roman{thmenumerate})\enspace}}
}
{}

\newcounter{xenumerate}   

\newcommand\pfitemx[1]{\par#1:}
\newcommand\pfitemref[1]{\pfitemx{\ref{#1}}}

\newcounter{steps}
\newcommand\stepx{\smallskip\noindent\refstepcounter{steps}%
 \emph{Step \arabic{steps}. }}

\newcommand{\refT}[1]{Theorem~\ref{#1}}
\newcommand{\refTs}[1]{Theorems~\ref{#1}}
\newcommand{\refC}[1]{Corollary~\ref{#1}}
\newcommand{\refL}[1]{Lemma~\ref{#1}}
\newcommand{\refR}[1]{Remark~\ref{#1}}
\newcommand{\refS}[1]{Section~\ref{#1}}
\newcommand{\refSS}[1]{Section~\ref{#1}}
\newcommand{\refStep}[1]{Step~\ref{#1}}
\newcommand{\refP}[1]{Proposition~\ref{#1}}

\newcommand{\refE}[1]{Example~\ref{#1}}

\newcommand\XREM[1]{\relax}

\begingroup
  \divide\count255 by 60
  \multiply\count255 by -60
  \advance\count255 by \time
  \ifnum \count255 < 10 \xdef\klockan{\the\count1.0\the\count255}
  \else\xdef\klockan{\the\count1.\the\count255}\fi
\endgroup

\newcommand\nopf{\qed}   

\DeclareMathOperator*{\sumx}{\sum\nolimits^{*}}

\newcommand{\summo}{\sum_{m=0}^\infty}

\newcommand{\sumk}{\sum_{k=1}^\infty}

\newcommand{\sumin}{\sum_{i=1}^n}

\newcommand{\sumjd}{\sum_{j=1}^d}
\newcommand{\sumkd}{\sum_{k=1}^d}

\newcommand\set[1]{\ensuremath{\{#1\}}}

\newcommand\xpar[1]{(#1)}
\newcommand\bigpar[1]{\bigl(#1\bigr)}
\newcommand\Bigpar[1]{\Bigl(#1\Bigr)}

\newcommand\lrpar[1]{\left(#1\right)}

\newcommand\xcpar[1]{\{#1\}}

\newcommand\abs[1]{|#1|}
\newcommand\bigabs[1]{\bigl|#1\bigr|}
\newcommand\Bigabs[1]{\Bigl|#1\Bigr|}
\newcommand\biggabs[1]{\biggl|#1\biggr|}

\def\rompar(#1){\textup(#1\textup)}    
\newcommand\xfrac[2]{#1/#2}
\newcommand\xpfrac[2]{(#1)/#2}

\newcommand\parfrac[2]{\lrpar{\frac{#1}{#2}}}
\newcommand\bigparfrac[2]{\bigpar{\frac{#1}{#2}}}
\newcommand\Bigparfrac[2]{\Bigpar{\frac{#1}{#2}}}

\def\xexp(#1){e^{#1}}
\newcommand\ceil[1]{\lceil#1\rceil}
\newcommand\floor[1]{\lfloor#1\rfloor}

\newcommand\setn{\set{1,\dots,n}}

\newcommand\ntoo{\ensuremath{{n\to\infty}}}

\newcommand\xtoo{\ensuremath{{x\to\infty}}}

\newcommand\normx[2]{\|#2\|_{#1}}
\newcommand\norm[1]{\|#1\|_2}
\newcommand\normp[1]{\normx{p}{#1}}

\newcommand\punkt{.\spacefactor=1000}    
\newcommand\iid{i.i.d\punkt}    
\newcommand\ie{i.e\punkt}
\newcommand\eg{e.g\punkt}

\newcommand{\as}{a.s\punkt}
\newcommand{\aex}{a.e\punkt}

\newcommand{\tend}{\longrightarrow}
\newcommand\dto{\overset{\mathrm{d}}{\tend}}
\newcommand\pto{\overset{\mathrm{p}}{\tend}}
\newcommand\asto{\overset{\mathrm{a.s.}}{\tend}}
\newcommand\eqd{\overset{\mathrm{d}}{=}}

\newcommand\bbR{\mathbb R}

\newcommand\bbZ{\mathbb Z}

\newcounter{CC}
\newcounter{cc}

\newcommand\E{\operatorname{\mathbb E{}}}
\renewcommand\P{\operatorname{\mathbb P{}}}

\newcommand\Var{\operatorname{Var}}
\newcommand\Cov{\operatorname{Cov}}

\newcommand\Be{\operatorname{Be}}
\newcommand\Ge{\operatorname{Ge}}

\newcommand\ga{\alpha}

\newcommand\gd{\delta}
\newcommand\gD{\Delta}

\newcommand\gam{\gamma}

\newcommand\gs{\sigma}
\newcommand\gS{\Sigma}
\newcommand\gss{\sigma^2}

\newcommand\eps{\varepsilon}

\renewcommand\phi{\xxx}  

\newcommand\cA{\mathcal A}

\newcommand\cE{\mathcal E}
\newcommand\cF{\mathcal F}

\newcommand\cS{{\mathcal S}}

\newcommand\cW{\mathcal W}

\newcommand\tF{\tilde F}

\newcommand\tS{{\tilde S}}

\newcommand\tU{{\widetilde U}}
\newcommand\tV{\tilde V}
\newcommand\tW{\widetilde W}

\newcommand\tZ{{\tilde Z}}

\newcommand\tf{\tilde f}
\newcommand\td{\tilde d}
\newcommand\tmu{\tilde \mu}

\newcommand\ett[1]{\boldsymbol1\xcpar{#1}}

\newcommand\etta{\boldsymbol1}

\newcommand\qw{^{-1}}
\newcommand\qww{^{-2}}
\newcommand\qq{^{1/2}}
\newcommand\qqw{^{-1/2}}

\newcommand\intoi{\int_0^1}

\newcommand\oi{\ensuremath{[0,1]}}

\newcommand\ooo{[0,\infty)}

\newcommand\setoi{\set{0,1}}

\newcommand\dd{\,\mathrm{d}}

\newcommand\lhs{left-hand side}
\newcommand\rhs{right-hand side}

\newcommand\nj{_{n,j}}
\newcommand\mj{_{m,j}}
\newcommand\Doo{D\ooo}
\newcommand\gDa{\gD a}
\newcommand\ux{\widehat U}
\newcommand\WW{\mathbf W}
\newcommand\oT{[0,T]}
\newcommand\oas{o_{\mathrm{a.s.}}}
\newcommand\oasx{o}
\newcommand\fXXd{f(X_1,\dots,X_d)}
\newcommand\FF{\widehat F}
\newcommand\xx[1]{^{(#1)}}
\newcommand\fx{f_*}
\newcommand\fS{\mathfrak S}
\newcommand\NN[1]{N_{#1}}
\newcommand\flnx{\floor{n(x)}}
\newcommand\ZZ{\widehat Z}
\newcommand\BB{\mathbf B}
\newcommand\UU{U^*}

\newcommand\nona{nonarithmetic}
\newcommand\Roo{R_\infty}
\newcommand\xxx{\gD x}
\newcommand\SSS{S^*}
\newcommand\Uoi{U(0,1)}

\newcommand\perm{\relax}

\newcommand\permB{\perm{231}, \perm{312}}

\newcommand\permAAA{\perm{231},\perm{312}, \perm{321}}

\newcommand\xoo{_1^\infty}
\newcommand\ELL{L}
\newcommand\Nxn[1]{N_{#1,n}}
\newcommand\Ngsn{\Nxn\gs}

\newcommand{\Holder}{H\"older}

\newcommand\CS{Cauchy--Schwarz}
\newcommand\CSineq{\CS{} inequality}

\begin{document}

\begin{abstract} 
We extend a functional limit theorem for symmetric $U$-statistics
[Miller and Sen, 1972] 
to asymmetric $U$-statistics, and use this to show some 
renewal theory results for asymmetric $U$-statistics.

 Some applications are given.
\end{abstract}

\maketitle

\section{Introduction}\label{S:intro}

Let $X,X_1,X_2,\dots,$ be an \iid{} sequence of random variables taking
values in an arbitrary measurable space $S=(S,\cS)$. (In most cases,
$S=\bbR$
or perhaps $\bbR^k$, or a Borel subset of one of these, but 
we can just as well consider the general case.)
Furthermore, let $d\ge1$ and let $f:S^d\to \bbR$ be a given measurable function.
We then define the (real-valued) random variables
\begin{equation}
  \label{U}
U_n=U_n(f):=\sum_{1\le i_1<\dots<i_d\le n} f\bigpar{X_{i_1},\dots,X_{i_d}},
\qquad n\ge0.
\end{equation}
We call $U_n$ a \emph{$U$-statistic}, 
following \citet{Hoeffding}.

\begin{remark}
Many authors, including
\citet{Hoeffding},
normalize $U_n$ by dividing
the sum in \eqref{U}
  by $\binom nd$,
the number of terms in it;
the traditional  definition 
(which assumes $n\ge d$)
is thus in our notation $U_n/\binom nd$.
We find it more convenient for our purposes
to use the unnormalized version above.
\end{remark}

It is common,  following \citet{Hoeffding}, to
assume that $f$ is a symmetric function of its $d$ variables.
In this case, the order of the variables does not matter, and we can in
\eqref{U} sum
over all sequences $i_1,\dots,i_d$ of $d$ distinct elements of $\setn$,
up to an obvious factor of $d!$. 
(\cite{Hoeffding} gives both versions.)
Conversely, if we sum over all such sequences, we may without loss of
generality
assume that $f$ is symmetric.
However, in the present paper we consider the general case of \eqref{U}
without assuming symmetry, which we for
emphasis may call \emph{asymmetric $U$-statistics}. 
One of the
purposes of this paper is to generalize a result
by \cite{MillerSen} on functional convergence
from the symmetric case to the general, asymmetric case.
We then use this result to derive some renewal theory results for the
sequence $U_n$.
One motivation for this is some applications to random restricted
permutations, see \refS{Sex}.

Univariate limit results, \ie, limits in distribution of $U_n$ after
suitable normalization, are well-known also in the asymmetric case, see 
\eg{}  \cite[Chapter 11.2]{SJIII}. 
The possibility of functional limits is
briefly mentioned in \cite[Remark 11.25]{SJIII},  
and a special case ($d=2$ and $f$ antisymmetric) was studied in \cite{SJ22},
see \refE{E22}; However, we are not aware of functional limit theorems in
the generality of the present paper.

The main results are stated in \refS{Smain}.
The proofs are given in \refS{Spf}; they use standard methods, in particular
the decomposition and projection method of \citet{Hoeffding},
but some complications arise in the asymmetric case; this includes
applications to random restricted permutations that gave the initial
motivation to write the present paper.
Some examples and applications are discussed in \refS{Sex}.
We end with some further comments and open problems
in \refS{Sadd}; this includes  
more comments on the relation between the symmetric and asymmetric cases.

The results in the present paper focus on the non-degenerate case, where the
covariance matrix   $\gS=(\gs_{ij})$ defined by \eqref{gsij} below is
non-zero. In the degenerate case when $\gS=0$, the result still holds but
are less interesting, since the obtained 
  limits in \eg{} \refT{T1} are degenerate. 
See \refR{Rdeg} for  further comments on the degenerate case.

\section{Some notation}\label{Snot}

We consider as in the introduction, unless otherwise said, some given \iid{}
random variables $X_i\in S$ and
a given function $f:S^d\to\bbR$. 
In particular,
$d\ge1$ is fixed,
and we therefore often  omit it from
the notation.

We assume throughout $f(X_1,\dots,X_d)\in L^1$ (and usually $L^2$), and define
\begin{equation}\label{mu}
  \mu:=\E \fXXd.
\end{equation}
We study $U_n=U_n(f)$ defined by \eqref{U}.
Let
\begin{equation}
  \label{U*}
U^*_n=U^*_n(f):=\max_{1\le m\le n}|U_m(f)|.
\end{equation}

We use $\normp\,$ for the $L^p$-norm: $\normp{Y}:=\xpar{\E Y^p}^{1/p}$ for any
random variable $Y$ and $p>0$, 
and $\normp{f}:=\normp{ f(X_1,\dots,X_d)}$ (and similarly
for other functions).

$\cF_n$ is the $\gs$-field generated by $X_1,\dots,X_n$.

If we consider a limit as \ntoo{}, and $a_n$ is a given sequence, 
then $\oas(a_n)$ denotes a sequence of random variables $R_n$ such that
$R_n/a_n\asto0$. This extends to other limits such as \xtoo, \emph{mutatis
  mutandis}. 

$C$ denotes positive constants that may change from one occurence to the
next; they may depend on $d$ (or $\td$)
but not on $f$ or $n$ or other variables. 
Similarly, $C_f$ denote constants that may depend on $f$,
$C_p$ denotes constants that may depend on the parameter $p$ (and
$d$), and so on.

\section{Main results} \label{Smain}

\subsection{Limit theorems}

For completeness, we begin with the law of large numbers, extending the
result by \citet{HoeffdingLLN} to the asymmetric case.

\begin{theorem}\label{TLLN}
  Suppose that $\fXXd\in L^1$. 
Then, as \ntoo,
  \begin{equation}\label{tlln}
    U_n/\binom nd \asto \mu.
  \end{equation}
\end{theorem}

Next we state a functional limit theorem,
extending the theorem by  \citet{MillerSen} for
the symmetric case.
 We use the space $\Doo$ with the
usual Skorohod topology, see \eg{} \cite[Appendix A2]{Kallenberg};
recall that convergence in $\Doo$ to a continuous limit is equivalent to
uniform convergence on any compact interval $\oT$.
We define the $d\times d$
matrix $\gS=(\gs_{ij})$ by
\begin{equation}\label{gsij}
  \gs_{ij}:=\Cov\bigpar{f_i(X),f_j(X)}
=\E\bigpar{f_i(X)f_j(X)},
\qquad i,j=1,\dots,d,
\end{equation}
with $f_i,f_j$ defined by \eqref{fi} below.
Let  $\WW(t):=\bigpar{W_1(t),\dots,W_d(t)}$, $t\ge0$, be a continuous
$d$-dimensional Gaussian process with $\WW(0)=0$ and 
stationary independent increments
\begin{equation}\label{WW}
\WW(s+t)-\WW(s)\sim N\bigpar{0,t\Sigma}.  
\end{equation}
Note that each component $W_j$ is a standard Brownian motion up to a factor
$\gs_{jj}\qq$, and that we can represent $\WW$ as
$\WW(t)=\Sigma\qq\BB(t)$, where $\BB(t)$ is a $d$-dimensional standard
Brownian motion.
Define also the functions
\begin{equation}\label{psi}
\psi_j(s,t)=\psi_{j;d}(s,t):=\frac{1}{(j-1)!\,(d-j)!}s^{j-1}(t-s)^{d-j}.
\end{equation}

We extend $U_n$ defined by \eqref{U}
to a function of a real variable by $U_x:=U_{\floor x}$,
$x>0$. 
(We tacitly do the same for other sequences later.)

\begin{theorem}
  \label{T1}
Suppose that $f(X_1,\dots,X_d)\in L^2$.
Then, as \ntoo,
  \begin{equation}\label{t1}
\frac{ U_{nt}-n^dt^d\mu/d!}{n^{d-1/2}} \dto Z_t, 
\qquad t\ge0,
  \end{equation}
in $\Doo$, where $Z_t$ is a continuous centered Gaussian process
that can be defined as
\begin{equation}\label{Z}
Z_t:=\sumjd  \int_0^t\psi_j(s,t)\dd W_j(s).
\end{equation}
Equivalently, $Z_t$ has the covariance function, for $0\le s\le t$,
{\multlinegap=0pt\begin{multline}
\label{t1cov}
      \Cov(Z_s,Z_t)
=\sum_{i,j=1}^d \gs_{ij} 
  \int_0^s  \psi_i(u,s)\psi_j(u,t)\dd u
\\
=
\sum_{i,j=1}^d \frac{\gs_{ij} }{(i-1)!\,(j-1)!\,(d-i)!\,(d-j)!}
  \int_0^s  u^{i+j-2}(s-u)^{d-i}(t-u)^{d-j}\dd u.
\end{multline}}

Moreover, \eqref{t1} holds jointly for several functions $f\xx{k}$, possibly
with different $d\xx k$, with limits given by \eqref{Z}, where the
corresponding $W_j\xx k$ together form a Gaussian process with stationary
independent increments given by the covariances
\begin{equation}\label{xul}
  \Cov\bigpar{W\xx k_i(s),W_j\xx \ell(t)}
= \Cov\bigpar{f_i\xx k(X),f_j\xx\ell(X)}\cdot(s\land t).
\end{equation}
\end{theorem}

The It\^o integrals in \eqref{Z} can by 
\eqref{psi} be written as linear combinations of $t^k\int_0^t s^{d-1-k}\dd
W_j(s)$ with $0\le k\le d-j$; 
thus $Z_t$ is well-defined and continuous for $t\ge0$, with $Z_0=0$. 
These stochastic integrals can also by integration by parts be expressed as
Riemann integrals of continuous stochastic processes, see \eqref{crux}.

Note that the final integral in \eqref{t1cov} is elementary, for any given
$i,j,d$, and that the covariance function in \eqref{t1cov} is a homogeneous
polynomial in $s$ and $t$ of degree $2d-1$.

\begin{example}\label{E2}
  In the case $d=2$, we obtain from \eqref{t1cov}, still for $0\le s\le t$,
  \begin{equation}\label{e2}
    \Cov(Z_s,Z_t)=
\tfrac12 \bigpar{\gs_{11}+\gs_{12}}s^2t
+\tfrac16\bigpar{2\gs_{22}-\gs_{11}-\gs_{12}}s^3.
  \end{equation}
\end{example}

\begin{remark}
By \eqref{psi} and the binomial theorem,
\begin{equation}\label{psisum}
  \sumjd \psi_j(s,t)=\frac{t^{d-1}}{(d-1)!}.
\end{equation}
  In the symmetric case, all $f_i$ are equal and thus all $\gs_{ij}$ are
  equal, see \eqref{gsij}. Hence, \eqref{t1cov} simplifies by \eqref{psisum}
  to 
\begin{equation}\label{t1cov=}
  \Cov(Z_s,Z_t)
=\gs_{11}   \int_0^s  \frac{s^{d-1}}{(d-1)!}\,\frac{t^{d-1}}{(d-1)!}\dd u
=\frac{\gs_{11}}{(d-1)!^2}s^{d}t^{d-1} .
\end{equation}
Equivalently, $t^{-(d-1)}Z_t$ is $\gs_{11}\qq(d-1)!\qw B_t$ for a standard
Brownian motion $B_t$. This recovers the result by \citet{MillerSen} for
the symmetric case.
Note that our general result \refT{T1} is similar to the symmetric case,
with a continuous Gaussian limit process, but that the covariance function
in general is more complicated, as seen for $d=2$ in \eqref{e2}, and that 
the limit thus is not  a Brownian motion.
\end{remark}

By restricting attention to $t=1$, we obtain the following univariate
limit,
shown in  \cite[Corollary 11.20]{SJIII}.
\begin{corollary}\label{C1}
Suppose that $f(X_1,\dots,X_d)\in L^2$.
Then, as \ntoo,
\begin{equation}\label{c1}
  \frac{U_n-\binom nd \mu}{n^{d-1/2}} \dto N\bigpar{0,\gss},
\end{equation}
where
\begin{equation}\label{c1var}
  \begin{split}
      \gss&:=
\lim_\ntoo \frac{\Var(U_n)}{n^{2d-1}}
=
\Var(Z_1)
\\&
\phantom:=
\sum_{i,j=1}^d
\frac{(i+j-2)!\,(2d-i-j)!}{(i-1)!\,(j-1)!\,(d-i)!\,(d-j)!\,(2d-1)!}\gs_{ij}
.
  \end{split}
\end{equation}
Moreover,
\begin{equation}\label{c1=0}
  \gss=0 \iff f_i(X)=0 \text{ \as{} for every $i=1,\dots,d$}.
\end{equation}
\end{corollary}

\begin{example}
  For $d=1$, \refC{C1} reduces to the Central Limit Theorem; indeed,
  \eqref{c1var} then yields $\gss=\gs_{11}$.

For $d=2$, \eqref{c1var} yields
\begin{equation}\label{c1var2}
  \gss=\frac{\gs_{11}+\gs_{12}+\gs_{22}}{3}.
\end{equation}
\end{example}

\subsection{Renewal theory}

For $x>0$, 
let 
\begin{align}
\NN-(x)&:=\sup\set{n\ge0:U_n\le x},\label{NN-}
\\
\NN+(x)&:=\inf\set{n\ge0:U_n>x}.  \label{NN+}
\end{align}
Note that if $f\ge0$, then $\NN+(x)=\NN-(x)+1$, but if $f$ attains negative
values, then $\NN-(x)>\NN+(x)$ is possible. Most of our results apply to
both $\NN+$ and $\NN-$; we then use $\NN\pm$ to denote any of them.

The results above easily imply some renewal theorems for $U$-statistics
generalizing well-known results for $S_n$ (i.e., the case $d=1$).
We begin with a law of large numbers.
\begin{theorem}
  \label{TNN}
Suppose that $\fXXd\in L^1$ and $\mu>0$. 
Then \as{} $\NN\pm(x)<\infty$ for every $x<\infty$,
and 
\begin{align}\label{tnn}
\frac{\NN\pm(x)}{x^{1/d}} \asto \parfrac{d!}{\mu}^{1/d}
\qquad  \text{as \xtoo}.
\end{align}
\end{theorem}

Assuming $f\in L^2$, we obtain also a central limit theorem for $\NN\pm$.

\begin{theorem}\label{TR}
 Suppose that $\fXXd\in L^2$ and $\mu>0$. 
Then, as \xtoo,
\begin{equation}\label{tr}
  \frac{\NN\pm(x)-\xpar{d!/\mu}^{1/d}x^{1/d}}{x^{1/2d}}
\dto N\Bigpar{0,\bigpar{\xfrac{d!}{\mu}}^{2+1/d}d\qww{ \gss}},
\end{equation}
where $\gss$ is given by \eqref{c1var}.
\end{theorem}

A situation that is common in application is to stop when when one process
(such as our $U_n$) reaches a threshold, and then look at the value of
another process, say $\tU_n$. For standard renewal theory, \ie{} the case
$d=1$ in our setting, this was studied in \cite{SJ50}; we extend the main
result there to (asymmetric) $U$-statistics.
We consider as above an \iid{}  sequence $X_1,X_2,\dots$ with values in $S$, 
but we now have two functions $f:S^d\to\bbR$ and $\tf:S^{\td}\to\bbR$, where
the numbers of variables $d$ and $\td$ may be different.
We use notations as above for both $f$ and $\tf$, with
$\,\tilde{}\,$ to denote variables defined by $\tf$, for example
$\tU_n:=U_n(\tf)$ and $\tmu:=\E \tf$; we furthermore assume that the
Gaussian processes 
$W_i(t)$ and $\tW_j(t)$ 
have the joint distribution specified by \eqref{xul} (with obvious
notational changes), and thus \eqref{t1} holds jointly for $f$ and $\tf$
with limits $Z_t$ and $\tZ_t$.

\begin{theorem}
  \label{TVtau}
\begin{thmenumerate}
\item \label{TVtauas}
Suppose that $\fXXd\in L^1$,
$\tf(X_1,\dots,X_{\td})\in L^1$
and  $\mu>0$. 
Then, as \xtoo,
\begin{equation}\label{tvtau0}
  \begin{split}
  \frac{\tU_{\NN\pm(x)}}{x^{\td/d}}
\asto    
\frac{\tmu}{\td!}\Bigparfrac{d!}{\mu}^{\td/d}.
  \end{split}
\end{equation}

\item \label{TVtaud}
Suppose that $\fXXd\in L^2$,
$\tf(X_1,\dots,X_{\td})\in L^2$
and  $\mu>0$. 
Then, as \xtoo,
\begin{equation}\label{tvtau}
\frac{\tU_{\NN\pm(x)}-\bigparfrac{d!}{\mu}^{\td/d}\frac{\tmu}{\td!} x^{\td/d}}
{x^{\td/d-1/2d}} 
\dto N\bigpar{0,\gam^2},
\end{equation}
where, with $(Z_1,\tZ_1)$ as in \refT{T1},
\begin{equation}\label{tvtau2}
\gam^2:=\Bigparfrac{d!}{\mu}^{(2\td-1)/d}
\Var\Bigpar{\tZ_1-\frac{(d-1)!\,\tmu}{(\td-1)!\,\mu}Z_1}.
\end{equation}
\item \label{TVtau=0}
Assume the conditions in  \ref{TVtaud}.
If $\td\ge d$, then
$\gam^2=0$ if and only if
\begin{equation}\label{ll}
  \tf_i(X) = \frac{\tmu}{\mu} 
\sum_j \frac{\binom{d-1}{j-1}\binom{\td-d}{i-j}} {\binom{\td-1}{i-1}} f_j(X)
\text{ a.s.}, 
\quad i=1,\dots,\td.
\end{equation}
If $\td<d$, then $\gam^2=0$ if and only if \eqref{ll} holds with $f,d,\mu$ and
$\tf,\tf,\tmu$ interchanged
(and $\tmu\neq0$ unless all $\tf_j(X)=0$ a.s.)

In particular, if $\td=d$, then 
\begin{equation}\label{ll1}
  \gam^2=0 \iff \mu\tf_i(X)=\tmu f_i(X) \quad a.s.,\qquad i=1,\dots,d,
\end{equation}
and if $d=1$, then 
\begin{equation}\label{ll2}
  \gam^2=0 \iff \mu\tf_i(X)=\tmu f_1(X) \quad a.s.,\qquad i=1,\dots,\td.
\end{equation}
\end{thmenumerate}
\end{theorem}

\begin{remark}\label{RVtau}
  \refT{TR} can be regarded as a special case of \refT{TVtau} with $\td=1$
  and $\tf(X)\equiv 1$.
\end{remark}

The asymptotic variance $\gam^2$ in \refT{TVtau}
can easily be calculated exactly using
\eqref{Z}, 
\eqref{xul} and \eqref{psi}, but a general formula seems more messy than
illuminating, and we state only the special case $d=1$.
(In this case, $U_n$ is the standard partial sum $\sumin f(X_i)$.)

\begin{theorem}\label{CVtau}
Suppose that $f(X)\in L^2$,
$\tf(X_1,\dots,X_{\td})\in L^2$
and  $\mu>0$. 
Then, as \xtoo,
\begin{equation}\label{cvtau}
\frac{\tU_{\NN\pm(x)}-{\mu}^{-\td}{\tmu}{\td!}\qw x^{\td}}
{x^{\td-1/2}} 
\dto N\bigpar{0,\gam^2},
\end{equation}
where
\begin{align}\label{cvtau2}
    \gam^2
&:=
{\mu}^{1-2\td}
\sum_{i,j=1}^{\td}
\frac{(i+j-2)!\,(2\td-i-j)!}{(i-1)!\,(j-1)!\,(\td-i)!\,(\td-j)!\,(2\td-1)!}
\Cov\bigpar{\tf_i(X),\tf_j(X)}
\notag\\&\qquad
-2\frac{{\mu}^{-2\td}\tmu}{(\td-1)!\,\td!}\sum_{i=1}^{\td}
\Cov\bigpar{f(X),\tf_i(X)}
\notag\\&\qquad
+ \frac{{\mu}^{-2\td-1}\tmu^2}{(\td-1)!^2}\Var\bigpar{f(X)}.
\end{align}
Moreover, 
\begin{equation}\label{ll22}
  \gam^2=0 \iff \mu\tf_i(X)=\tmu (f(X)-\mu) \quad a.s.,\qquad i=1,\dots,\td.
\end{equation}
\end{theorem}

Continue to assume that $d=1$, and assume for simplicity that
$Y:=f(X)\ge0$ a.s.
Thus $U_n(f)=S_n(f):=\sum_1^n Y_i$ is a renewal process, 
and its \emph{overshoot} (\emph{residual life time}) is
\begin{equation}\label{R}
  R(x):=U_{\NN+(x)}-x>0.
\end{equation}
A classical result, see \eg{} \cite[Theorem 2.6.2]{Gut-SRW},
says that if $0<\mu<\infty$, then $R(x)$ converges in distribution.
Recall that (the distribution of) $Y$ has \emph{span} $d>0$ if $Y\in d\bbZ$
a.s., and $d$ is maximal with this property, and that (the distribution of) $Y$
is \nona{} if no such $d$ exists.
\begin{proposition}[e.g.\ \cite{Gut-SRW}]\label{PR}
Let $R(x)$ be given by \eqref{R}, and assume that $f(X)\ge0$ a.s.
  \begin{romenumerate}
  \item 
If $f(X)$ is \nona, then $R(x)\dto\Roo$ as \xtoo, with
    \begin{equation}
      \label{overa}
\P\bigpar{\Roo\le y}
=
\frac{1}{\mu}\int_0^y \P\bigpar{f(X)>s}\dd s,
\qquad y\ge0.
    \end{equation}
  \item 
If $f(X)$ has span $d>0$, then $R(x)\dto\Roo$  as \xtoo{} with
$x\in d\bbZ$, with
    \begin{equation}
      \label{overd}
\P\bigpar{\Roo=kd}
=\frac{d}{\mu} \P\bigpar{f(X)\ge kd},
\qquad k\ge1.
    \end{equation}
  \end{romenumerate}
\nopf
\end{proposition}

This classical result may be combined with \refT{CVtau} as follows.

\begin{theorem}\label{TO}
  Suppose in addition to the assumptions of \refT{CVtau} that $f(X)\ge0$
  a.s.
Let $\Roo$ be as in \refP{PR}.
  \begin{romenumerate}
  \item \label{TOa}
If $f(X)$ is \nona, then \eqref{cvtau} and $R(x)\dto \Roo$  hold jointly as
\xtoo.    
  \item \label{TOb}
If $f(X)$ has span $d>0$, then \eqref{cvtau} and $R(x)\dto\Roo$  hold jointly as
\xtoo{} with $x\in d\bbZ$.
\item \label{TOc}
If $f(X)$ is integer-valued, then for every fixed integer $k\ge1$,
\eqref{cvtau} holds also conditioned on $R(x)=k$, for integers $x=\ntoo$.
Moreover, \eqref{cvtau} holds also conditioned on $U_{\NN-(x)}=x$,
as $x=\ntoo$.
(We consider only $x$ such that we condition on an event of positive
probability.)
\end{romenumerate}
\end{theorem}

Note that in \ref{TOc}, the event $U_{\NN-(x)}=x$ holds if and only if some
partial sum $U_n:=\sum_1^n f(X_i)=x$.

\begin{remark}
  If $d=\td=1$, \eqref{cvtau2} reduces to
$\gam^2=\mu^{-3}\Var\bigpar{\mu\tf(X)-\tmu f(X)}$, 
as shown in \cite[Theorem 3]{SJ50}.
\end{remark}

\subsection{Moment convergence}\label{SSmoments}
In \refC{C1}, we have convergence of the second moment in \eqref{c1}, and
trivially also of the first moment.
We have also convergence of higher moments, provided we assume the
corresponding integrability of $f$.

\begin{theorem}
  \label{TUp}
Suppose that $\fXXd\in L^p$ with $p\ge2$. 
Then, \eqref{c1} holds with convergence of all moments and absolute moments
of order $\le p$.
\end{theorem}

For moment convergence in the renewal theory theorems, we assume for
simplicity that $f$ and $\tf$ have finite moments of all orders;
see also \refR{Rmom}.
(For the case $d=\td=1$, see \eg{} 
\cite{SJ52}, \cite{SJ50}, and 
\cite[Section 3.8 and Theorem 4.2.3]{Gut-SRW}.)

\begin{theorem}\label{TRp}
Suppose that $\fXXd\in L^p$ for every $p<\infty$, and thet $\mu>0$.
Then, \eqref{tnn} and
\eqref{tr} hold with convergence of all moments and absolute moments.
In particular, as \xtoo, 
\begin{align}\label{trpe}
\E\NN\pm(x)&\sim \Bigparfrac{d!}{\mu}^{1/d}x^{1/d},
\\
  \Var\NN\pm(x)&\sim \bigpar{\xfrac{d!}{\mu}}^{2+1/d}d\qww{ \gss} x^{1/d}. 
\label{trpv}
\end{align}
\end{theorem}

\begin{theorem}\label{TVp}
Suppose that $\fXXd\in L^p$ and
$\tf(X_1,\dots,X_{\td})\in L^p$ for every $p<\infty$,
and  that $\mu>0$. 
Then, \eqref{tvtau0} and
\eqref{tvtau} hold with convergence of all moments and absolute moments.
In particular, as \xtoo,
\begin{align}\label{tvpe}
\E\tU_{\NN\pm(x)}&\sim \frac{\tmu}{\td!}\Bigparfrac{d!}{\mu}^{\td/d}
x^{\td/d},
\\
  \Var\tU_{\NN\pm(x)}&\sim \gam^2 x^{(2\td-1)/d}. \label{tvpv}
\end{align}
\end{theorem}

\begin{theorem}\label{TVp1}
Let $d=1$. Suppose that $f(X)\in L^p$ and
$\tf(X_1,\dots,X_{\td})\in L^p$ for every $p<\infty$,
and  that $\mu>0$. 
\begin{romenumerate}
\item \label{TVp1a}
Then, \eqref{cvtau}
holds with convergence of all moments and absolute moments.

\item \label{TVp1b}
If furthermore $f(X)$ is integer-valued and $f(X)\ge0$, then 
\ref{TVp1a} holds also conditioned on $R(x)=k$ or on $U_{\NN-(x)}=x$ as in
\refT{TO}\ref{TOc}.
\end{romenumerate}
\end{theorem}

\section{Proofs}\label{Spf}

\subsection{Limit theorems}\label{SSpflimit}

The method used by \citet{Hoeffding} and many later papers is a
decomposition, which in the asymmetric case is as follows.
Assume that $f(X_1,\dots,X_d)\in L^2$ and
define, recalling \eqref{mu},
\begin{align}
f_i(x)&:=\E \bigpar{f(X_1,\dots,X_d)\mid X_i=x}-\mu 
\notag
\\&\phantom:
= \E f\bigpar{X_1,\dots,X_{i-1},x,X_{i+1},\dots,X_d}-\mu,
\label{fi}
\\
\fx(x_1,\dots,x_d)
&:=
f(x_1,\dots,x_d) - \mu - \sumjd f_j(x_j).
\label{h}
\end{align}
(In general, these are defined only \aex, but that is no problem.)
Then, by the definition \eqref{U},
\begin{equation}\label{unf}
  \begin{split}
U_n(f)&=\binom nd \mu +\sumjd 
\sum_{1\le i_1<\dots<i_d\le n} f_j\bigpar{X_{i_j}}
+ U_n(\fx)
\\
&=\binom nd \mu 
+\sumjd \sumin \binom{i-1}{j-1}\binom{n-i}{d-j}f_j\bigpar{X_{i}}
+ U_n(\fx).
  \end{split}
\end{equation}
We consider the three terms in \eqref{unf} separately.
The first is a constant, and we shall see that the third term is negligible,
so the main term is the second term.

\begin{remark}
  The decomposition \eqref{unf} may be continued to higher terms by
  expanding $\fx$ further, see \eg{} \cite{Hoeffding} for the symmetric case
  and \cite[Chapter 11.2]{SJIII} in general; this is important when treating
  degenerate cases, 
  see \refR{Rdeg}, but for our purposes we have no need of this.
\end{remark}

For the second term, we define for convenience, for
$1\le j\le d$ and $n\ge1$,
\begin{align}
  a_{n,j}(i)&:=\binom{i-1}{j-1}\binom{n-i}{d-j}, 
\qquad 1\le i\le n,
\label{anj}
\\
\gDa\nj(i)&:=a\nj(i+1)-a\nj(i),
\qquad 1\le i< n. \label{bnj}
\end{align}

Recall $\psi(s,t)$ defined in \eqref{psi}, and let $\psi'(s,t)$ denote
$\frac{\partial}{\partial s}\psi(s,t)$.

\begin{lemma}\label{LA}
Uniformly for all $n$, $j$, $i$ such that the variables are defined,
\begin{align}
a\nj(i)=\psi_j(i,n)+O\bigpar{n^{d-2}},\label{la}
\\  
\gDa\nj(i)=\psi_j'(i,n)+O\bigpar{n^{d-3}}.\label{la'}
\end{align}
In particular, $a\nj(i)=O\bigpar{n^{d-1}}$
and $\gDa\nj(i)=O\bigpar{n^{d-2}}$.
Furthermore
(for $d\le2$),
any error term $O(n^{-1})$ or $O(n^{-2})$ here
vanishes identically. 
\end{lemma}
\begin{proof}
  By \eqref{anj}, for $1\le i\le n$,
  \begin{equation}
    \begin{split}
 a\nj
&=\frac{i^{j-1}+O(n^{j-2})}{(j-1)!}\cdot \frac{(n-i)^{d-j}+O(n^{d-j-1})}{(d-j)!}
\\&
=\frac{i^{j-1}(n-i)^{d-j}}{(j-1)!\,(d-j)!}+O\bigpar{n^{d-2}},      
    \end{split}
  \end{equation}
which is \eqref{la}.
Similarly, for $1\le i< n$,
with $\binom k{-1}=0$,
  \begin{equation}
    \begin{split}
\gDa\nj
&=\lrpar{\binom{i}{j-1}-\binom{i-1}{j-1}}\binom{n-i}{d-j}
\\&\hskip8em
+\binom{i}{j-1}\lrpar{\binom{n-i-1}{d-j}-\binom{n-i}{d-j}}
\\
&=\binom{i-1}{j-2}\binom{n-i}{d-j}
-\binom{i}{j-1}\binom{n-i-1}{d-j-1}
\\
&=\frac{(j-1)i^{j-2}(n-i)^{d-j}-(d-j)i^{j-1}(n-i)^{d-j-1}+O\bigpar{n^{d-3}}}
{(j-1)!\,(d-j)!}
\\
&=\psi_j'(i,n)
+O\bigpar{n^{d-3}}.
    \end{split}
  \end{equation}
\end{proof}

We now take case of the second term in \eqref{unf}.
\begin{lemma}\label{Lux}
  Let
  \begin{equation}\label{lux1}
    \ux\nj:= \sumin \binom{i-1}{j-1}\binom{n-i}{d-j}f_j\bigpar{X_{i}}
= \sumin a\nj(i) f_j\bigpar{X_{i}}.
  \end{equation}
Then, as \ntoo, with $W_j$ as in \eqref{WW},
\begin{equation}\label{lux2}
 n^{-(d-1/2)} \ux_{nt,j}\dto \int_0^t\psi_j(u,t)\dd W_j(u),
\qquad t\ge0,
\end{equation}
in $\Doo$, jointly for $j=1,\dots,d$.
\end{lemma}
\begin{proof}
   Let for any function $g:S\to\bbR$,
\begin{equation}\label{sn}
  S_n(g):=U_n(g):=\sumin g(X_i).
\end{equation}
Then, by  \eqref{sn}, \eqref{bnj}, and a summation by parts,
\begin{equation}\label{pi}
  \begin{split}
\ux\nj&=
\sumin a\nj(i)f_j(X_{i}) 
=\sumin a\nj(i)\bigpar{S_{i}(f_j)-S_{i-1}(f_j)}
\\&
=a\nj(n)S_n(f_j)-\sum_{i=1}^{n-1} \gDa\nj(i)S_i(f_j).
  \end{split}
\end{equation}

By \eqref{fi}, $\E f_j(X)=0$, and furthermore $f_j(X)\in L^2$.
Hence, by Donsker's theorem, 
\begin{equation}\label{donsker}
  n\qqw S_{nt}(f_j)\dto W_j(t),
\end{equation}
in $\Doo$, jointly for $j=1,\dots,d$, where $W_j$ are 
continuous centered Gaussian processes as in \eqref{WW}.

By the Skorohod coupling theorem \cite[Theorem 4.30]{Kallenberg}, 
we may assume that the convergence in \eqref{donsker} holds \as, and thus 
as \ntoo,
\begin{equation}\label{ddonsker}
n\qqw  S_{nt}(f_j)= W_j(t) + \oas\xpar{1},
\end{equation}
uniformly for $t\in\oT$ and all $j$, for every fixed $T<\infty$.
(Note that the error term here, $R_{n,j,t}$ say, is random; the uniformity
means that $\sup_{j\le d,\,t\le T}|R_{n,j,t}|\allowbreak\asto0$ for every $T$.)

Fix $T$, and let $m=ns$ with $s\le T$.
Then, by \eqref{pi}, \eqref{ddonsker} and \refL{LA}, uniformly for $s\in\oT$,
\begin{equation}\label{hux}
  \begin{split}
n\qqw\ux_{m,j}&
=a\mj(m) W_j(s) -\sum_{i=1}^{m-1} \gDa\mj(i) W_j(i/n)
+\oas\bigpar{n^{d-1}}
\\&
={\psi_j(m,m) W_j(s) -\sum_{i=1}^{m-1} \psi'_j(i,m) W_j(i/n)}
+\oas\bigpar{n^{d-1}}.
  \end{split}
\end{equation}
Furthermore, since $W_j$ is bounded and uniformly continuous on $\oT$, 
with $W_j(0)=0$, 
and
$\psi'_j(s,t)=O(t^{d-2})$,
$\psi''_j(s,t)=O(t^{d-3})$
for $0\le s\le t$,
\begin{equation}\label{flux}
  \begin{split}
      \sum_{i=1}^{m-1} \psi'_j(i,m) W_j(i/n)&
=  \int_{0}^{m} \psi'_j(x,m) W_j(x/n)\dd x+\oasx\bigpar{m^{d-1}}
\\&
=  n\int_{0}^{s} \psi'_j(nu,ns) W_j(u)\dd u+\oasx\bigpar{n^{d-1}}
\\&
=  n^{d-1}\int_{0}^{s} \psi'_j(u,s) W_j(u)\dd u+\oasx\bigpar{n^{d-1}}
.  \end{split}
\end{equation}
An integration by parts yields (with stochastic integrals)
\begin{equation}\label{crux}
  \begin{split}
\int_{0}^{s} \psi'_j(u,s) W_j(u)\dd u
=   \psi_j(s,s)W_j(s)-\int_{0}^{s} \psi_j(u,s)\dd W_j(u)
  \end{split}
\end{equation}
and combining \eqref{hux}, \eqref{flux} and \eqref{crux} yields,
using $\psi_j(m,m)=n^{d-1}\psi_j(s,s)$,
\begin{equation}
  \begin{split}
n\qqw\ux_{ns,j}&=n\qqw\ux_{m,j}
=
n^{d-1}\int_{0}^{s} \psi_j(u,s)\dd W_j(u)+\oas\bigpar{n^{d-1}}
,
  \end{split}
\end{equation}
uniformly for $0\le s\le T$.
Since $T$ is arbitrary, 
this yields \eqref{lux2}, jointly for all $j$.
\end{proof}

To show that the final term in \eqref{unf} is negligible, we give another lemma.
Cf.\  \cite{Sproule1974} 
for similar results in the symmetric case.

\begin{lemma}
  \label{LU*}
Suppose that $\fXXd\in L^2$.
\begin{romenumerate}
\item \label{LU*1}
Then
\begin{equation}\label{lu*1}
  \E |U_n^*(f-\mu)|^2 \le C n^{2d-1}\norm{f}^2.
\end{equation}
\item \label{LU*2}
If furthermore $f_i=0$ for $i=1,\dots,d$, then
\begin{equation}\label{lu*2}
  \E |U_n^*(f-\mu)|^2 \le C n^{2d-2}\norm{f}^2.
\end{equation}
\end{romenumerate}
\end{lemma}

\begin{proof}
\pfitemref{LU*1}
  We introduce another decomposition of $f$ and $U_n$, which unlike the one
  in \eqref{fi}--\eqref{unf} focusses on the order of the arguments.
Let $\FF_0:=\mu$ and, for $1\le k\le d$,
\begin{align}
  \FF_k(x_1,\dots,x_k)&:=\E f(x_1,\dots,x_k,X_{k+1},\dots,X_d),\label{FF}
\\
F_k(x_1,\dots,x_k)&:=\FF_k(x_1,\dots,x_k)-\FF_{k-1}(x_1,\dots,x_{k-1}).\label{F}
\end{align}
In other words, 
$\FF_k(X_1,\dots,X_k):=\E\bigpar{f(X_1,\dots,X_d)\mid  X_1,\dots,X_k}$, and
thus 
$\FF_k(X_1,\dots,X_k)$, $k=0,\dots,d$, is a martingale, with
the martingale differences 
$F_k(X_1,\dots,X_k)$, $k=1,\dots,d$.
Hence,
\begin{equation}
  \label{EFk}
\E F_k(x_1,\dots,x_{k-1},X_k)=0.
\end{equation}

By \eqref{FF}--\eqref{F}, $f(x_1,\dots,x_d)-\mu=\sumkd F_k(x_1,\dots,x_k)$,
and thus
\begin{equation}
  \begin{split}
U_n(f-\mu)
&=\sumkd \sum_{i_1<\dots<i_k\le n} \binom{n-i_k}{d-k} F_k\xpar{X_{i_1},\dots,X_{i_k}} 
\\
&=\sumkd \sumin \binom{n-i}{d-k}\bigpar{U_i(F_k)-U_{i-1}(F_k)}
\\
&=U_n(F_d)+\sum_{k=1}^{d-1} \sum_{i=1}^{n-1} \binom{n-i-1}{d-k-1} U_i(F_k),
  \end{split}
\end{equation}
using a summation by parts and the identity 
$\binom{n-i}{d-k}-\binom{n-i-1}{d-k}=\binom{n-i-1}{d-k-1}$.
In particular,
\begin{equation}\label{kum}
  \begin{split}
|U_n(f-\mu)|
&\le |U_n(F_d)|+\sum_{k=1}^{d-1} \sum_{i=1}^{n-1} \binom{n-i-1}{d-k-1}
U^*_n(F_k)
\\
&= |U_n(F_d)|+\sum_{k=1}^{d-1} \binom{n-1}{d-k} U^*_n(F_k)
\le \sum_{k=1}^{d} n^{d-k} U^*_n(F_k).
  \end{split}
\end{equation}
Since the \rhs{} is weakly increasing in $n$, it follows that
\begin{equation}\label{kul}
U^*_n(f-\mu)
\le \sum_{k=1}^{d} n^{d-k} U^*_n(F_k).
\end{equation}

By the definition \eqref{U},
$\gD U_n(F_k):=U_n(F_k)-U_{n-1}(F_k)$ is a sum of $\binom{n-1}{k-1}$ terms
$F_k(X_{i_1},\dots,X_{i_{k-1}},X_n)$ that all have the same distribution,
and thus by Minkowski's inequality,
\begin{equation}\label{swab}
\E|\gD U_n(F_k)|^2 = \norm{\gD U_n(F_k)}^2
\le \binom{n-1}{k-1}^2\norm{F_k}^2 
\le n^{2k-2}\norm{f}^2.
\end{equation}
Furthermore,
it follows from \eqref{EFk} that $\E \bigpar{U_n(F_k)-U_{n-1}(F_k)\mid
  \cF_{n-1}}=0$, 
and thus $U_n(F_k)$, $n\ge0$, is a martingale.
Consequently, using \eqref{swab},
\begin{equation}\label{mb}
\E|U_n(F_k)|^2 = \sumin \abs{\E \gD U_i(F_k)}^2
\le n^{2k-1}\norm{f}^2
\end{equation}
and
Doob's inequality yields
\begin{equation}\label{mba}
\norm{U^*_n(F_k)}
\le C \norm{U_n(F_k)}
\le C n^{k-1/2}\norm{f}.
\end{equation}
Finally, \eqref{kul}, \eqref{mba} and Minkowski's inequality yield 
\begin{equation}\label{mbc}
  \norm{U^*_n(f-\mu)}
\le \sum_{k=1}^{d} n^{d-k} \norm{U^*_n(F_k)}
\le C n^{d-1/2}\norm{f},
\end{equation}
which yields \eqref{lu*1} by squaring.

\pfitemref{LU*2}
By \eqref{FF}--\eqref{F} and \eqref{fi}, 
\begin{equation}
  \E \bigpar{F_k(X_1,\dots,X_k)\mid X_k}
=   \E \bigpar{f(X_1,\dots,X_d)\mid X_k}-\E f =f_k(X_k).
\end{equation}
Hence, assuming $f_k=0$, 
\begin{equation}
  \label{lie}
  \E \bigpar{F_k(X_1,\dots,X_k)\mid X_k}=0.  
\end{equation}
It was seen in the proof of \ref{LU*1} that 
$\gD U_n(F_k)$ is a sum of $\binom{n-1}{k-1}$ terms
$F_k(X_{i_1},\dots,X_{i_{k-1}},X_n)$. It now follows from \eqref{lie} that
if $\set{i_1,\dots,i_{k-1}}$ and $\set{j_1,\dots,j_{k-1}}$ are two 
disjoint sets of indices, then, by first conditioning on $X_n$,
\begin{equation}
  \E\bigpar{F_k(X_{i_1},\dots,X_{i_{k-1}},X_n)F_k(X_{j_1},\dots,X_{j_{k-1}},X_n)}
=0.
\end{equation}
Hence, only the $O\bigpar{n^{2k-3}}$ pairs of index sets
$\set{i_1,\dots,i_{k-1}}$ and $\set{j_1,\dots,j_{k-1}}$ with at least one
common element contribute to $\E\bigpar{\gD U_n(F_k)}^2$, and we obtain, for
$1\le k\le d$, that \eqref{swab} is improved to
\begin{equation}\label{gul}
  \E |\gD U_n(F_k)|^2 \le C n^{2k-3}\norm{f}^2.
\end{equation}
(For $k=1$, $F_1=f_1=0$, and \eqref{gul} still holds.)
The result now follows as in \ref{LU*1},
see \eqref{mb}--\eqref{mbc},
by \eqref{gul}, Doob's  inequality, 
 \eqref{kul} and Minkowski's inequality.
\end{proof}

\begin{proof}[Proof of \refT{T1}]
  We use the decomposition \eqref{unf}, with $n$ replaced by $\floor{nt}$.
For the constant term, note that
$
  \binom {\floor{nt}}d \mu = n^dt^d\mu/d!+O\bigpar{n^{d-1}} 
$
when $t=O(1)$.

The second term in \eqref{unf} 
is $\sumjd \ux_{nt,j}$, using the notation in \eqref{lux1}, and
we use \refL{Lux};
\eqref{lux2} shows that this term divided by $n^{d-1/2}$ converges in $\Doo$
to $Z_t$ defined in \eqref{Z}.

For the third term, we apply \refL{LU*} to $\fx$. It follows from the
definition \eqref{h} that
$\mu_*:=\E \fx(X_1,\dots,X_d)=0$ and that, applying \eqref{fi} to $\fx$, 
$(\fx)_i=0$ for every $i\le d$.
Hence, \refL{LU*}\ref{LU*2} applies to $\fx$ and yields
\begin{equation}\label{qul}
  \E \abs{U_n^*(\fx)}^2\le C n^{2d-2}\norm{\fx}^2\le C n^{2d-2}\norm{f}^2.
\end{equation}
Let $T>0$ be fixed. Applying \eqref{qul} to $nT$, we see in particular that
$n^{-(d-1/2)}U_{nt}(\fx)\pto0$
uniformly on $\oT$.

Consequently, \eqref{t1} follows from \eqref{unf}.

Joint convergence for several functions $f\xx k$, with limits given by
\eqref{xul}, follows by the same proof, using joint convergence for all
$f\xx k_i$ in \eqref{donsker}.
\end{proof}

\begin{proof}[Proof of \refT{TLLN}]
We do this in several steps.

\stepx\label{stepL2} 
First, suppose that $\fXXd\in L^2$.  
We may assume $\mu=0$, and then \refL{LU*}\ref{LU*1} implies, for any
$N\ge1$,
\begin{equation}
\E \sup_{N\le n\le 2N}\bigpar{|U_n|/n^d}^2 \le N^{-2d} \E (U^*_{2N})^2 \le C
N^{-1}\norm{f}^2. 
\end{equation}
Summing over all $N=2^m$, $m=0,1,\dots$, we find
\begin{equation}
\E\summo \sup_{2^m\le n\le 2^{m+1}}\bigpar{|U_n|/n^d}^2 <\infty.
\end{equation}
Hence, \as{} the terms in the sum tend to 0, which implies $U_n/n^d\to 0$ and
thus $U_n/\binom nd\to 0=\mu$. This  proves \eqref{tlln} for $f\in L^2$.

\stepx\label{stepL1>} 
Assume now $f\in L^1$ and $f\ge0$. Define the truncation $f_M:=f\land
M$. Then $f_M\in L^2$ and \refStep{stepL2} shows that for every $M<\infty$,
a.s.,
\begin{equation}
  \liminf_\ntoo \frac{U_n(f)}{\binom nd} 
\ge \liminf_\ntoo \frac{U_n(f_M)}{\binom nd }
= \E f_M(X_1,\dots,X_d).
\end{equation}
Letting $M\to\infty$ yields 
$\liminf_\ntoo U_n(f)/\binom nd \ge \mu$ a.s.

\stepx Continue to assume $f\in L^1$ and $f\ge0$. For every permutation
$\pi\in\fS_d$, let $f_\pi(X_1,\dots,X_d):=f(X_{\pi(1)},\dots,X_{\pi(d)})$,
and let $F:=\sum_{\pi\in\fS} f_\pi$ and $g:=F-f=\sum_{\pi\neq id} f_\pi$.
Note that $f,g\in L^1$ with $f,g\ge0$; thus \refStep{stepL1>} 
applies to both $f$ and
$g$. Furthermore, $F=f+g$ is symmetric, so we have 
$U_n(F)/\binom nd\asto \E F:=\E F(X_1,\dots,X_d)$
by the theorem by \citet{HoeffdingLLN} for the symmetric case. (This case
has a simple reverse martingale proof, see \refR{Rreverse}.)
Consequently, a.s.,
\begin{equation}
  \limsup_\ntoo \frac{U_n(f)}{\binom nd}
=
  \lim_\ntoo \frac{U_n(F)}{\binom nd}
-
  \liminf_\ntoo \frac{U_n(g)}{\binom nd}
\le \E F - \E g = \mu.
\end{equation}
Combined with \refStep{stepL1>}, this shows \eqref{tlln} for every $f\in
L^1$ with $f\ge0$.

\stepx
The general case follows by linearity.
\end{proof}

We used for convenience the known symmetric case in this proof. An
alternative would be to use suitable truncations, similarly to the original
proof of the symmetric case by \citet{HoeffdingLLN}.

\begin{lemma}
\label{Lvar}
Suppose that $\fXXd\in L^2$. Then, as \ntoo,
with $Z_1$ defined by \eqref{Z},
\begin{equation}\label{lvar}
  \begin{split}
    \frac{\Var U_n}{n^{2d-1}} 
\to
  \gss
&:=
\Var Z_1
\\&
\phantom:=
\sum_{i,j=1}^d
\frac{(i+j-2)!\,(2d-i-j)!}{(i-1)!\,(j-1)!\,(d-i)!\,(d-j)!\,(2d-1)!}\gs_{ij}
.
  \end{split}
\end{equation}
\end{lemma}

\begin{proof}
  We may assume $\mu=0$. Then
  \begin{equation}
\Var U_n = \E U_n^2=
\sum_{i_1<\dots<i_d}\sum_{j_1<\dots<j_d}
  \E \bigpar{f(X_{i_1},\dots,X_{i_d})f(X_{j_1},\dots,X_{j_d})},
  \end{equation}
where all terms with $\set{i_1,\dots,i_d} \cap \set{j_1,\dots,j_d}
=\emptyset$ vanish. 
There are only $O\bigpar{n^{2d-2}}$ terms with
$|\set{i_1,\dots,i_d} \cap \set{j_1,\dots,j_d}|\ge2$,
so we concentrate on the case 
when, say, $i_k=j_\ell=i$, and all other indices are distinct.
Thus, using \eqref{fi} and the notation \eqref{anj} together with
\eqref{gsij}
and \refL{LA},
\begin{equation}
  \begin{split}
\E U_n^2
&=\sum_{k=1}^d\sum_{\ell=1}^d\sumin a_{n,k}(i)a_{n,\ell}(i) 
\E \bigpar{f_k(X_i)f_\ell(X_i)}+O\bigpar{n^{2d-2}}   
\\
&=\sum_{k=1}^d\sum_{\ell=1}^d\sumin \psi_{k}(i,n)\psi_{\ell}(i,n) \gs_{k\ell} 
+O\bigpar{n^{2d-2}}   
\\
&=\sum_{k=1}^d\sum_{\ell=1}^d \gs_{k\ell} \int_0^n \psi_{k}(x,n)\psi_{\ell}(x,n)\dd x 
+O\bigpar{n^{2d-2}}   
\\
&=n^{2d-1}\sum_{k=1}^d\sum_{\ell=1}^d \gs_{k\ell} \int_0^1
\psi_{k}(u,1)\psi_{\ell}(u,1)\dd u 
+O\bigpar{n^{2d-2}} .  
  \end{split}
\end{equation}
Consequently, by \eqref{t1cov}, 
\begin{equation}
  \frac{\Var U_n}{n^{2d-1}} \to
\sum_{k=1}^d\sum_{\ell=1}^d \gs_{k\ell} \int_0^1
\psi_{k}(u,1)\psi_{\ell}(u,1)\dd u 
=
\Var(Z_1).
\end{equation}
Furthermore, 
this equals the sum in \eqref{lvar}, as is seen  by taking $s=t=1$ in
\eqref{t1cov} 
and evaluating the resulting Beta integral.
\end{proof}

\begin{remark}
  Similarly, it follows more generally that
$\Cov\bigpar{U_{ns},U_{nt}}/n^{2d-1}\to\Cov(Z_s,Z_t)$ given by
\eqref{t1cov},
for any fixed $s,t\ge0$.
In other words, \eqref{t1} holds with convergence of second moments.
\end{remark}

\begin{proof}[Proof of \refC{C1}]
  The functional limit \eqref{t1} implies, since $Z_t$ is continuous,
  convergence (in distribution) for each fixed $t\ge0$.
Taking $t=1$ we obtain \eqref{c1} with $\gss=\Var Z_1$, which is evaluated
by \refL{Lvar}.

By \eqref{t1cov} and \eqref{gsij},
\begin{equation}
  \label{lasse}
  \begin{split}
\Var (Z_1)
&=\sum_{i,j=1}^d\Cov\bigpar{f_i(X),f_j(X)} \intoi \psi_i(s,1) \psi_j(s,1)\dd s   
\\&
=\intoi \Var\Bigpar{\sum_{i=1}^d\psi_i(s,1)f_i(X)}\dd s
  \end{split}
\end{equation}
Hence, $\gss=0\iff \sum_{i=1}^d\psi_i(s,1)f_i(X)=0$ \as{} for (almost) every
$s\in\oi$,
which is equivalent to $f_i(X)=0$ \as{} for every $i$ since the polynomials
$\psi_i(s,1)$ are linearly independent.
\end{proof}

\subsection{Renewal theory}\label{SSpfRenewal}

\begin{proof}[Proof of \refT{TNN}]
Consider first $\NN-$.
Note that \refT{TLLN} and $\mu>0$ imply $U_n\to\infty$  a.s., and then
  $\NN-(x)<\infty$ for every $x$.

Furthermore, it is trivial that $\NN-(x)\to\infty$ as \xtoo.
Thus we may substitute $n=\NN-(x)$ in \eqref{tlln} and obtain
\begin{equation}\label{ollon}
  \frac{U_{\NN-(x)}}{\NN-(x)^d}
=
  \frac{U_{\NN-(x)}}{\binom{\NN-(x)}{d}}\cdot
\frac{\binom{\NN-(x)}{d}}{\NN-(x)^d}
\asto\frac{\mu}{d!}
\qquad \text{as \xtoo}.
\end{equation}
Furthermore, we also have, again by \eqref{tlln},
\begin{equation}\label{kollon}
  \frac{U_{\NN-(x)+1}}{\NN-(x)^d}
=
  \frac{U_{\NN-(x)+1}}{\binom{\NN-(x)+1}{d}}\cdot
\frac{\binom{\NN-(x)+1}{d}}{\NN-(x)^d}
\asto\frac{\mu}{d!}.
\end{equation}
By the definition of $\NN-(x)$, 
$U_{\NN-(x)}\le x < U_{\NN-(x)+1}$, and thus \eqref{ollon}--\eqref{kollon}
imply 
\begin{equation}\label{bollon}
  \frac{x}{\NN-(x)^d}
\asto\frac{\mu}{d!}
\qquad \text{as \xtoo}.
\end{equation}
which is equivalent to \eqref{tnn} for $\NN-$.

The proof for $\NN+$ is the same, using $U_{\NN+(x)-1}\le x< U_{\NN+(x)}$.
\end{proof}

\begin{proof}[Proof of \refT{TR}]

Again,   we consider $\NN-$;  the argument for $\NN+$ is the same.
Let
\begin{align}\label{nx}
    n(x)&:=\xpar{d!/\mu}^{1/d}x^{1/d},
\\
T(x)&:=\NN-(x)/\floor{n(x)}.\label{Tx}
\end{align}
As \xtoo, $n(x)\to\infty$ and thus \eqref{t1} implies
  \begin{equation}\label{t1x}
\frac{ U_{\flnx t}-(\flnx t)^d\mu/d!}{n(x)^{d-1/2}} \dto Z_t
\qquad \text{in }\Doo.
  \end{equation}
Furthermore, \eqref{tnn} implies
\begin{equation}\label{tx1}
T(x)\to 1  
\end{equation}
a.s., and thus in probability. Hence,
\eqref{t1x} and \eqref{tx1} hold jointly in distribution 
\cite[Theorem 4.4]{Billingsley}. 
Now, $(F,t)\mapsto F(t)$ is a
measurable mapping $D\ooo\times \ooo\to\bbR$ that is continuous at every
$(F,t)$ with $F$ continuous.
Hence, by \cite[Theorem 5.1]{Billingsley},
it follows from
the joint convergence in \eqref{t1x} and \eqref{tx1}, together with
continuity of $Z_t$, 
that we may substitute $t=T(x)$ in
\eqref{t1x} and 
obtain, as $\xtoo$,
\begin{equation}\label{dum}
\frac{  U_{\NN-(x)}-\NN-(x)^d \mu/d!}{n(x)^{d-1/2}}\dto Z_1.
\end{equation}
Taking instead $t=T_1(x):=(\NN-(x)+1)/\flnx$, we similarly obtain
\begin{align}\label{dee}
\frac{  U_{\NN-(x)+1}-\NN-(x)^d \mu/d!}{n(x)^{d-1/2}}
&= 
\frac{  U_{\NN-(x)+1}-(\NN-(x)+1)^d \mu/d!+O\bigpar{\NN-(x)^{d-1}+1}}{n(x)^{d-1/2}}
\notag\\& 
\dto Z_1,
\end{align}
using $(\NN-(x)^{d-1}+1)/n(x)^{d-1/2}\pto0$ by \eqref{tnn} and \eqref{nx}.
Since $U_{\NN-(x)}\le x< U_{\NN-(x)+1}$, \eqref{dum} and \eqref{dee}
together imply,
as \xtoo,
\begin{equation}\label{dumdee}
\frac{  x-\NN-(x)^d \mu/d!}{n(x)^{d-1/2}}\dto Z_1.
\end{equation}
Hence, recalling \eqref{nx}, 
\begin{align}
\frac{x}{n(x)^{d-1/2}}
\lrpar{
  \parfrac{\NN-(x)}{n(x)}^d-1} 
&= \frac{\NN-(x)^d\mu/d!-x}{n(x)^{d-1/2}}
\dto -Z_1.
\label{dumdum}
\end{align}
Furthermore, letting $T_2(x):=\NN-(x)/n(x)$, we have $T_2(x)\asto1$ by
\eqref{tnn}, and thus,  interpreting the quotients as $d$ when $T_2(x)=1$,
\begin{equation}\label{hack}
\frac{  \bigpar{\xfrac{\NN-(x)}{n(x)}}^d-1 }
{  \bigpar{\xfrac{\NN-(x)}{n(x)}}-1 }
=
  \frac{T_2(x)^d-1}{T_2(x)-1} \asto d.
\end{equation}
Dividing \eqref{dumdum} by \eqref{hack} yields
\begin{equation}\label{don}
\frac{x}{n(x)^{d-1/2}}
\lrpar{
\frac{\NN-(x)}{n(x)}-1 }
\dto -\frac{1}{d}Z_1.
\end{equation}
Since
\begin{equation}\label{don32}
\frac{\NN-(x)-n(x)}{x^{1/2d}} 
= \parfrac{n(x)}{x^{1/d}}^{d+1/2}
\frac{x}{n(x)^{d-1/2}}
\lrpar{
\frac{\NN-(x)}{n(x)}-1 },
\end{equation}
\eqref{don} and \eqref{nx} imply
\begin{equation}\label{don2}
\frac{\NN-(x)-n(x)}{x^{1/2d}} 
\dto-\parfrac{d!}{\mu}^{1+1/2d}d\qw Z_1,
\end{equation}
which yields \eqref{tr}, since $Z_1\sim N(0,\gss)$ by \refL{Lvar}.
\end{proof}

\begin{proof}[Proof of \refT{TVtau}]
\pfitemref{TVtauas}
By \refT{TLLN} for $\tf$ and \eqref{tnn},
\begin{equation}
  \begin{split}
  \frac{\tU_{\NN\pm(x)}}{x^{\td/d}}
=
  \frac{\tU_{\NN\pm(x)}}{\NN\pm(x)^{\td}}
  \frac{\NN\pm(x)^{\td}}{x^{\td/d}}
\asto    
\frac{\tmu}{\td!}\Bigparfrac{d!}{\mu}^{\td/d}.
  \end{split}
\end{equation}

\pfitemref{TVtaud}
Define again $n(x)$ and $T(x)$ by \eqref{nx}--\eqref{Tx}.
We have joint convergence in \eqref{t1} for $f$ and $\tf$, and 
thus, as \xtoo, \eqref{t1x} holds jointly with
  \begin{equation}\label{t2x}
\frac{ \tU_{\flnx t}-(\flnx t)^{\td}\tmu/\td!}{n(x)^{\td-1/2}} \dto \tZ_t
\qquad \text{in }\Doo.
  \end{equation}
By \eqref{tx1} and
the argument in the proof of \refT{TR},
now using the mapping $(F,\tF,t)\mapsto(F(t),\tF(t))$ that maps
$\Doo\times\Doo\times\ooo\to\bbR^2$,
it follows that \eqref{dum} holds jointly with
\begin{equation}\label{tdum}
\frac{\tU_{\NN-(x)}-\NN-(x)^{\td} \tmu/\td!}{n(x)^{\td-1/2}}\dto \tZ_1.
\end{equation}
Furthermore, 
\eqref{dum} and \eqref{dumdee} together with $U_{\NN-(x)}\le x$ imply
\begin{equation}
  \frac{  x-U_{\NN-(x)}}{n(x)^{d-1/2}}\pto 0.
\end{equation}
Consequently, \eqref{dumdee} and \eqref{tdum} hold jointly.
The argument in the proof of \refT{TR} now holds with every convergence in
distribution holding jointly with \eqref{tdum}.
Hence, 
\eqref{tdum} holds jointly with
\eqref{don2}, which implies, see \eqref{hack}
and \eqref{nx},
\begin{equation}\label{uvb}
  \frac{\NN-(x)^{\td}-n(x)^{\td}}{n(x)^{\td-1/2}}
=
  \frac{\bigpar{\NN-(x)/n(x)}^{\td}-1}{\bigpar{\NN-(x)/n(x)}-1}
  \frac{\NN-(x)-n(x)}{n(x)^{1/2}}
\dto -\td\, \frac{d!}{\mu}\,d\qw Z_1.
\end{equation}
Consequently, 
\eqref{tdum} and \eqref{uvb} hold jointly,
and thus
\begin{equation}\label{tvtaux}
 \frac{\tU_{\NN-(x)}-\xfrac{n(x)^{\td}\tmu}{\td!}}{n(x)^{\td-1/2}} 
\dto 
\ZZ:=
\tZ_1
- \frac{(d-1)!\,\tmu}{(\td-1)!\,\mu}Z_1
.
\end{equation}
We obtain \eqref{tvtau}--\eqref{tvtau2} by substituting the definition
\eqref{nx} of $n(x)$. 

\pfitemref{TVtau=0}
By \eqref{tvtau2}, 
$\gam^2=0\iff \Var\bigpar{\mu(\td-1)!\,\tZ_1-\tmu(d-1)!\, Z_1}=0$, and
arguing as in \eqref{lasse}, and recalling \eqref{xul}, this is equivalent
to
\begin{equation}
  \Var\Bigpar{\mu(\td-1)!\sum_{i=1}^{\td}\psi_{i;\td}(s,1)\tf_i(X)
-\tmu(d-1)!\sum_{j=1}^{d}\psi_{j;d}(s,1)f_j(X)}=0
\end{equation}
for (almost) every $s\in\oi$, 
and by the definition \eqref{psi}, this is the same as
\begin{equation}\label{tuna}
  \mu\sum_{i=1}^{\td}\binom{\td-1}{i-1}s^{i-1}(1-s)^{\td-i}\tf_i(X)
=
\tmu\sum_{j=1}^{d}\binom{d-1}{j-1}s^{j-1}(1-s)^{d-j}f_j(X)
\end{equation}
a.s., for every $s$.

If $\td\ge d$, multiply the \rhs{} of \eqref{tuna} by $(s+1-s)^{\td-d}=
\sum_{k=0}^{\td-d}\binom{\td-d}{k} s^k(1-s)^{\td-d-k}$, which equals 1,
and identify the coefficients of $s^{i-1}(1-s)^{\td-i}$ on both sides; this
yields \eqref{ll}. Conversely, \eqref{ll} implies
\eqref{tuna} by the same argument.

The case $\td<d$ follows by the symmetry in \eqref{tuna}.

The special cases \eqref{ll1} and \eqref{ll2} are immediate consequences of
\eqref{ll}.
\end{proof}

\begin{proof}[Proof of \refT{CVtau}]
Take $d=1$ in \refT{TVtau}\ref{TVtaud}.
To obtain the formula \eqref{cvtau2} for $\gam^2$, we use \eqref{tvtau2} and
note first that $\Var(\tZ_1)$ is given by \eqref{c1var}, mutatis mutandis,
  which yields the first term on the \rhs{} of \eqref{cvtau2}.
Furthermore, \eqref{c1var} yields also, with $d=1$, 
$\Var (Z_1)=\gs_{11}=\Var(f(X))$, yielding the third term.
Finally, note that when $d=1$,
\eqref{psi} yields $\psi_{1;1}(s,t)=1$, and thus
\eqref{Z} yields
$Z_t=W(t)$; consequently, using \eqref{Z} and \eqref{xul} and a standard
Beta integral,
\begin{equation}
  \begin{split}
\Cov\bigpar{\tZ_1,Z_1}
&
=\sum_{j=1}^{\td}\Cov\Bigpar{\intoi\psi_{j;\td}(s,1)\dd \tW_j(s),\intoi\dd W(s)}    
\\&
=\sum_{j=1}^{\td} \intoi\psi_{j;\td}(s,1)\Cov\bigpar{\tf_j(X),f(X)} \dd s
\\&
=\sum_{j=1}^{\td}\frac{1}{\td!} \Cov\bigpar{\tf_j(X),f(X)} .
  \end{split}
\end{equation}
This yields the second term on the \rhs, and completes the proof.
\end{proof}

\begin{proof}[Proof of \refT{TO}]
Let (for $x\ge2$, say)
$x_-:=x-\ln x$ in the \nona{} case, and $x_-:=d\floor{(x-\ln x)/d}$ if $f(X)$
  has span $d>0$; also, in the latter case, consider only $x\in d\bbZ$.

First, run the process until the stopping time $\NN+(x_-)$.
Let 
\begin{equation}
  \xxx:=x-U_{\NN+(x_-)} = x-x_--R(x_-).
\end{equation}
As \xtoo, $R(x-)\dto\Roo$ by \refP{PR}, and $x-x_-\ge\ln x\to\infty$;
hence $\xxx\pto\infty$. In particular, with probability tending to 1 as
\xtoo, $\xxx\ge0$.

Restart the process after $\NN+(x_-)$ and continue until $\NN+(x)$.
Since $\NN+(x_-)$ is a stopping time, this continuation is independent of
what happened up to $\NN+(x_-)$, and thus it can be regarded as a renewal
process $\SSS_n$ 
starting at 0 and running to $\NN+(\xxx)$; in particular, the overshoot
$R^*(\xxx)$ of this renewal process equals the overshoot $R(x)$ of the
original one.
Here $\xxx$ is random,
but independent of the renewal process $\SSS_n$, and
since $\xxx\pto\infty$, \refP{PR}
implies that the overshoot $R(x)=R^*(\xxx)\dto\Roo$.
Furthermore, this holds conditioned on any events $\cE(x_-)$ that depend on the
original process up to $\NN+(x_-)$, provided $\liminf_{\xtoo}\P(\cE(x_-))>0$.

Denote the \lhs{} of \eqref{cvtau} by $\tV(x)$.
By \eqref{cvtau},
$\tV(x_-)\dto N(0,\gam^2)$ as \xtoo. Fix $a,b\in\bbR$ and let
$\cE(x_-):=\set{\tV(x_-)\le a}$. 
It then follows from the argument above that, as \xtoo,
\begin{equation}
  \begin{split}
      \P\bigpar{\tV(x_-)\le a,\,R(x)\le b}
&=  \P\bigpar{R(x)\le b\mid \tV(x_-)\le a}\P\bigpar{\tV(x_-)\le a}
\\&
\dto \P\bigpar{\Roo\le b}\P\bigpar{N(0,\gam^2)\le a}.
  \end{split}
\end{equation}
Consequently, $\tV(x_-)$ and $R(x)$ converge jointly, with independent
limits given by \eqref{cvtau} and \eqref{overa}--\eqref{overd}.

It remains only to replace by $\tV(x_-)$ by $\tV(x)$.
First, since $x_-=x-O(\ln x)$ it follows that $\tV(x_-)\dto N(0,\gam^2)$ is
equivalent to
\begin{equation}\label{cvtau-}
\frac{\tU_{\NN\pm(x_-)}-{\mu}^{-\td}{\tmu}{\td!}\qw x^{\td}}
{x^{\td-1/2}} 
\dto N\bigpar{0,\gam^2},
\end{equation}
Hence, \eqref{cvtau-} and $R(x)\dto\Roo$ hold jointly, with independent
limits.

Next, suppose first that $\tf(X_1,\dots,X_{\td})\ge0$.
Then, $\tU_{\NN\pm(x)}\ge\tU_{\NN\pm(x_-)}$ a.s., and thus \eqref{cvtau} and
\eqref{cvtau-} imply
\begin{equation}\label{cvtau00}
\frac{\tU_{\NN\pm(x)}-\tU_{\NN\pm(x_-)}}
{x^{\td-1/2}} 
\pto 0.
\end{equation}
By linearity, \eqref{cvtau00} holds for arbitrary $\tf\in L^2$.
Finally, \eqref{cvtau00} and  \eqref{cvtau-} imply \eqref{cvtau},
and hence
 \eqref{cvtau00} and the joint convergence of \eqref{cvtau-} and
$R(x)\dto\Roo$ imply the joint convergence of \eqref{cvtau} and
$R(x)\dto\Roo$,
proving \ref{TOa} and \ref{TOb}.

For \ref{TOc},  let $d$ be the span of $f(X)$, and assume first $d=1$.
Note that $\P(\Roo=k)=0\iff\P\bigpar{f(X)\ge k}=0$
by \eqref{overd}, and then \eqref{R} implies $R(x)\le f(X_{\NN+(x)})<k$ 
\as{} for every $x$; 
hence we only consider $k$ such that $\P(\Roo=k)>0$, and
the first part of \ref{TOc} follows from \ref{TOb}.

If the span $d>1$, then $R(x)=k$ implies $x+k=U_{\NN+(x)}\equiv 0\pmod d$
and thus $x\equiv -k\pmod d$, so we consider only $x\in -k+d\bbZ$.
Let $k_0:=d\ceil{k/d}$ and $\gD:=k_0-k\in[0,d-1]$.
Then $x-\gD\equiv x+k\equiv0\pmod d$, and thus, since $S_n(f)\in d\bbZ$,
$\NN+(x)=\NN+(x-\gD)$ and
$R(x-\gD)=U_{\NN+(x)}-x+\gD=R(x)+\gD$; hence
\begin{equation}
  R(x)=k 
\iff
R(x-\gD)=k+\gD=k_0.
\end{equation}
Hence, we may replace $x$ and $k$ by $x-\gD$ and $k_0$, 
and thus it suffices to consider $x,k\in d\bbZ$, but then we can reduce to
the case $d=1$ by replacing $f(X)$ by $f(X)/d$.

Finally, for an integer $n$,
$U_{\NN-(n)}=n \iff R(n-1)=1$. 
Hence, \eqref{cvtau} with $x=n-1$ holds as \ntoo, also
conditioned on $U_{\NN-(n)}=n$.
The argument above showing \eqref{cvtau00} shows also that
$\xpfrac{\tU_{\NN\pm(n)}-\tU_{\NN\pm(n-1)}}{n^{\td-1/2}} \pto 0$ as \ntoo,
and it follows that 
 \eqref{cvtau} with $x=n$ holds as \ntoo,
conditioned on $U_{\NN-(n)}=n$.  
\end{proof}

\subsection{Moment convergence}\label{SSpfmoments}

We turn to proving the theorems on moment convergence in \refSS{SSmoments},
and begin by extending \refL{LU*} to higher absolute moments.

\begin{lemma}
  \label{LUp}
Suppose that $\fXXd\in L^p$ with $p\ge2$. 
Then
\begin{equation}\label{lup}
  \E |U_n^*(f-\mu)|^p \le C_p n^{p(d-1/2)}\normp{f}^p.
\end{equation}
\end{lemma}

\begin{proof}
  We use the same decomposition as in the proof of \refL{LU*}.
Note that, by Jensen's inequality, $\normp{\FF_k}\le\normp{f}$, 
and thus, 
  \begin{equation}
    \normp{F_k}\le2\normp{f},
\qquad 1\le k\le d.
  \end{equation}
Hence, Minkowski's inequality yields, as in \eqref{swab},
\begin{equation}\label{swabp}
\E|\gD U_n(F_k)|^p = \normp{\gD U_n(F_k)}^p
\le \binom{n-1}{k-1}^p\normp{F_k}^p 
\le C_p n^{pk-p}\normp{f}^p.
\end{equation}
Consequently, the Burkholder inequalities \cite[Theorem 10.9.5(i)]{Gut}
applied to the martingale $U_n(F_k)$
yield, using also \Holder's inequality,
\begin{align}
\label{mbap}
  \E|U^*_n(F_k)|^p
&\le C_p \E \Bigpar{\sumin \abs{\gD U_i(F_k)}^2}^{p/2}
\le C_p \E \Bigpar{n^{p/2-1}\sumin \abs{\gD U_i(F_k)}^p}
\notag\\&
=C_p n^{p/2-1}\sumin \E\abs{\gD U_i(F_k)}^p
\le C_p n^{pk-p/2}\normp{f}^p.
\end{align}
Equivalently,
\begin{align}
\label{mbapp}
  \normp{U^*_n(F_k)}
\le C_p n^{k-1/2}\normp{f}.
\end{align}
Finally, \eqref{kul}, \eqref{mbapp} and Minkowski's inequality yield 
\begin{equation}\label{mbcp}
  \normp{U^*_n(f-\mu)}
\le \sum_{k=1}^{d} n^{d-k} \normp{U^*_n(F_k)}
\le C_p n^{d-1/2}\normp{f},
\end{equation}
which is \eqref{lup}.
\end{proof}

We shall also use the following standard result, stated in detail and proved for
convenience and completeness.
\begin{lemma}\label{Lui}
  Let $\set{V_\ga:\ga\in \cA}$ be a set of random variables, and let $0<p<q$.
Suppose that for every $\eps>0$ there exist decompositions
$V_\ga=V_\ga'+V_\ga''$ and a $B_\eps<\infty$ 
such that, for every $\ga\in \cA$, 
$\normx{q}{V_\ga'}\le B_\eps$  and $\normp{V_\ga''}\le\eps$. Then the set
\set{|V_\ga|^p} is uniformly integrable.
\end{lemma}
\begin{proof}
  If $\gd>0$ and $\cE$ is any event with $\P(\cE)\le\gd$, then, using
  \Holder's inequality,
  \begin{align}
        \E \bigpar{|V_\ga|^p\etta_{\cE}}
&\le C_p \E \bigpar{|V_\ga'|^p\etta_{\cE}}
  + C_p \E\bigpar{|V_\ga''|^p\etta_{\cE}}
\notag\\&
\le C_{p} \normx{q}{V_\ga'}^p\P(\cE)^{1-p/q} + C_p\normx{p}{V_\ga''}^p
\notag\\&
\le C_{p} B_\eps^{p} \gd^{1-p/q} + C_p \eps^p. 
  \end{align}
Since $\eps$ is arbitrary, this can be made arbitrarily small, uniformly in
$\ga$, by choosing first choosing $\eps$ and then $\gd$ small.
\end{proof}

\begin{proof}[Proof of \refT{TUp}]
Denote the \lhs{} of \eqref{c1} by $V_n$. Then $\E|V_n|^p$ is bounded by 
 \refL{LUp}.
This implies
convergence of all moments and absolute moments of
 order $<p$ in \eqref{c1} by standard arguments, but is not by itself enough
to include  moments of order $p$.  
Thus we use a truncation: 
let $M>0$ and let $f=f'+f''$ with $f':=f\ett{|f|\le M}$.
This yields a corresponding decomposition
$V_n=V_n'+V_n''$.
Let  $\eps_M:=\normp{f''}$. Then
\begin{equation}\label{emm}
  \eps_M:=\normp{f \ett{|f|>M}} \to0
\quad \text{as } M\to\infty.
\end{equation}
\refL{LUp} yields
\begin{equation}\label{viip}
\normp{V_n''}\le C_p\normp{f''}= C_p\eps_M
\end{equation}
and also, using $2p$ instead of $p$, 
\begin{equation}\label{vip}
  \normx{2p}{V_n'}^{2p}\le C_p\normx{2p}{f'}^{2p}
=C_p\E|f'|^{2p}\le C_p M^p \E|f|^p.
\end{equation}
\eqref{emm}--\eqref{vip} show that the conditions of  \refL{Lui} are
satisfied; hence, \set{|V_n|^p} is uniformly integrable, and the
result follows from \eqref{c1}.
\end{proof}

We use another simple lemma.
\begin{lemma}\label{Lpq}
  Suppose that, 
for each $x\ge1$,
$V(x)$ is a non-negative random variable and $v(x)>0$ is
  deterministic.
  \begin{romenumerate}
  \item \label{Lpqa}
If $p\ge1$,  $q\ge1$    and, for some function $h(x)>0$,
\begin{equation}\label{lpqpq}
  \E |V(x)^q-v(x)^q|^p = O\bigpar{v(x)^{pq} h(x)^p},
\qquad x\ge1,
\end{equation}
then
\begin{equation}\label{lpqp}
  \E |V(x)-v(x)|^p = O\bigpar{v(x)^{p} h(x)^p},
\qquad x\ge1.
\end{equation}
  \item \label{Lpqb}
Conversely, if \eqref{lpqp} holds for every $p\ge1$ and $h(x)\le1$, 
then \eqref{lpqpq}
holds for every $p,q\ge1$.
  \end{romenumerate}
\end{lemma}

\begin{proof}
  \pfitemref{Lpqa}
If $a>b\ge0$, then 
\begin{equation}
  a^q-b^q=a^q\bigpar{1-(b/a)^q}
\ge a^q\bigpar{1-(b/a)}=a^{q-1}(a-b)=\max\set{a,b}^{q-1}(a-b).
\end{equation}
Hence, by symmetry, for all $a,b\ge0$,
\begin{equation}
  \abs{a^q-b^q}\ge
\max\set{a,b}^{q-1}|a-b|.
\end{equation}
In particular,
\begin{equation}
  \abs{V(x)^q-v(x)^q}\ge
v(x)^{q-1}|V(x)-v(x)|,
\end{equation}
and thus \eqref{lpqpq} implies \eqref{lpqp}.

\pfitemref{Lpqb}
If $V(x)\le 2v(x)$, then, by the mean value theorem,
$|V(x)^q-v(x)^q|\le C_q v(x)^{q-1}|V(x)-v(x)|$. Thus, using \eqref{lpqp},
\begin{multline}\label{august}
    \E\bigpar{\abs{V(x)^q-v(x)^q}^p\ett{V(x)\le 2v(x)}}
\\
\le C_{p,q} v(x)^{pq-p}\E|V(x)-v(x)|^p 
= O\bigpar{v(x)^{pq}h(x)^p}.
\end{multline}
On the other hand, if $V(x)>2v(x)$, then
$|V(x)^q-v(x)^q|\le V(x)^q\le 2^q|V(x)-v(x)|^q$. 
Thus, using \eqref{lpqp} with $p$ replaced by $pq$,
\begin{multline}\label{lotta}
  \E\bigpar{\abs{V(x)^q-v(x)^q}^p\ett{V(x)> 2v(x)}}
\\
\le C_{p,q} \E|V(x)-v(x)|^{pq} 
= O\bigpar{v(x)^{pq}h(x)^{pq}}
.
\end{multline}
The result follows by \eqref{august} and \eqref{lotta}.
\end{proof}

\begin{proof}[Proof of \refT{TRp}]
  As usual, we consider for definiteness $\NN-(x)$. 
By the definition \eqref{NN-}, 
$U_{\NN-(x)}\le x< U_{\NN-(x)+1}$. Hence,
\begin{multline}
-\UU_{\NN-(x)}(f-\mu)
\le U_{\NN-(x)}(f-\mu)
\le x-\binom{\NN-(x)}{d}\mu  
\\
\le \UU_{\NN-(x)+1}(f-\mu)+ C_f \NN-(x)^{d-1}
\end{multline}
and thus
\begin{align}\label{bab}
\Bigabs{x-\NN-(x)^{d}\frac{\mu}{d!}}  
\le \UU_{\NN-(x)+1}(f-\mu)+ C_f \NN-(x)^{d-1}.
\end{align}

Suppose throughout $x\ge1$, and recall $n(x)$ defined by \eqref{nx}. 
By \eqref{bab} and \refL{LUp}, for any $p>0$ and any $A\ge1$,
\begin{align}
&  \E\Bigpar{\Bigabs{x-\NN-(x)^{d}\frac{\mu}{d!}}^p\ett{\NN-(x)\le An(x)}}  
\notag\\&\qquad
\le C_p \E |\UU_{An(x)+1}(f-\mu)|^p+ C_{p,f} \bigpar{A n(x)}^{p(d-1)}
\notag\\&\qquad
\le C_{p,f} \bigpar{A n(x)}^{p(d-1/2)}
= C_{p,f} A^{p(d-1/2)}x^{p(1-1/2d)}.
\label{bai}
\end{align}
Furthermore, for any constant $A\ge2$,
$\NN-(x)\ge An(x)$ implies 
$\NN-(x)^d\frac{\mu}{d!}-x\ge (A^d-1)x\ge\frac12A^dx$.
Hence,  for any $p\ge0$ and $q>0$, using \eqref{bai},
\begin{align}
&\E\Bigpar{\Bigabs{x-\NN-(x)^{d}\frac{\mu}{d!}}^p\ett{An(x)<\NN-(x)\le 2An(x)}}
\notag\\&\qquad
\le
C_q A^{-dq}x^{-q}
 \E\Bigpar{\Bigabs{x-\NN-(x)^{d}\frac{\mu}{d!}}^{p+q}\ett{\NN-(x)\le 2An(x)}}
\notag\\&\qquad
\le
C_{p,q,f} A^{(p+q)(d-1/2)-dq}x^{(p+q)(1-1/2d)-q}
\notag\\&\qquad
=
C_{p,q,f} A^{p(d-1/2)-q/2}x^{p(1-1/2d)-q/2d}.
\label{baj}
\end{align}
Choosing $q:=2dp$, we obtain by summing \eqref{bai} with $A=2$ and
\eqref{baj} with $A=2^k$, $k=1,2,\dots$,  
for every $p>0$,
\begin{align}\label{bak}
\E\bigabs{n(x)^d-\NN-(x)^{d}}^p
&= C_{p,f}\E\Bigabs{x-\NN-(x)^{d}\frac{\mu}{d!}}^p
\notag\\&
\le
C_{p,f} x^{p(1-1/2d)}
+C_{p,f} \sumk 2^{-kp/2}x^{-p/2d}
\notag\\&
\le
C_{p,f} x^{p(1-1/2d)}.
\end{align}
By \refL{Lpq}\ref{Lpqa}, with $q=d$ and $h(x):=x^{-1/2d}$,
\eqref{bak} implies, for $p\ge1$, 
\begin{align}\label{bal}
\E\bigabs{n(x)-\NN-(x)}^p
\le
C_{p,f} x^{p/2d}.
\end{align}
This shows that if $Y(x)$ denotes the \lhs{} of \eqref{tr}, then
$\E|Y(x)|^p\le C_{p,f}$ for $x\ge1$. 
By standard arguments \cite[Chapter 5.4--5]{Gut}, this implies uniform
integrability of $|Y(x)|^r$ for any $r<p$, and thus by \eqref{tr}
convergence of moments of order $<p$. Since $p$ is arbitrary, 
convergence of arbitrary moments in \eqref{tr}
follows.

Moment convergence in \eqref{tnn} is an immediate corollary. 
Alternatively,
\eqref{bak} implies 
\begin{equation}
  \label{Nmom}
\E \bigpar{\NN-(x)^{dp}} = O\bigpar{x^p},
\qquad x\ge1,
\end{equation}
for every
fixed $p>0$, which implies moment convergence in \eqref{tnn} by the same
uniform integrability argument.
\end{proof}

\begin{proof}[Proof of \refT{TVp}]
Recall again the definition \eqref{nx} of  $n(x)$, and suppose again $x\ge1$.
We decompose the numerator in \eqref{tvtau}:
\begin{equation}\label{bb}
  \tU_{\NN\pm(x)}-\Bigparfrac{d!}{\mu}^{\td/d}\frac{\tmu}{\td!} x^{\td/d}
=
  U_{\NN\pm(x)}(\tf-\tmu)
+\frac{\tmu}{\td!}\bigpar{\NN\pm(x)^{\td}-n(x)^{\td}}
+O\bigpar{\NN\pm(x)^{\td-1}}.
\end{equation}

For the first term on the \rhs{} of \eqref{bb}, we argue similarly 
to the proof of \refT{TRp}.
First, for any $A\ge2$, by \refL{LUp},
\begin{multline}\label{bbe}
\E\Bigpar{\bigabs{  U_{\NN\pm(x)}(\tf-\tmu)}^p\ett{\NN\pm(x)\le A n(x)}}
\le 
\E\bigabs{U^*_{A n(x)}(\tf-\tmu)}^p
\\
\le C_{p,\tf} \bigpar{An(x)}^{p(\td-1/2)}
= C_{p,f,\tf} \bigpar{Ax^{1/d}}^{p(\td-1/2)}.
\end{multline}
Furthermore, for any $q>0$, taking $p=0$ in \eqref{baj},
\begin{equation}\label{bbn}
  \P\bigpar{An(x) < N(x)\le 2An(x)}
\le C_{q,f} \bigpar{Ax^{1/d}}^{-q/2}.
\end{equation}
Consequently, using the \CSineq, \eqref{bbe}--\eqref{bbn},
and choosing $q:=4(p\td+1)$,
\begin{equation}\label{bbg}
  \begin{split}
&\E\Bigpar{\bigabs{\tU_{\NN\pm(x)}(\tf-\tmu)}^p\ett{A n(x)<\NN\pm(x)\le 2A n(x)}}
\\&\hskip4em
\le
\Bigpar{
\E\Bigpar{\bigabs{\tU_{\NN\pm(x)}(\tf-\tmu)}^{2p}\ett{\NN\pm(x)\le 2A n(x)}}}\qq
\\&\hskip10em \times
\P\bigpar{A n(x)<\NN\pm(x)\le 2A n(x)}\qq
\\&\hskip4em
\le C_{p,f,\tf} \bigpar{Ax^{1/d}}^{p(\td-1/2)-q/4}
\le C_{p,f,\tf} A\qw x^{p(\td-1/2)/d}.    
  \end{split}
\end{equation}
Summing \eqref{bbe} for $A=2$ and \eqref{bbg} for $A=2^k$, $k=1,2,\dots$, we
obtain
\begin{equation}\label{bbh}
  \begin{split}
\E\bigabs{\tU_{\NN\pm(x)}(\tf-\tmu)}^p
\le
C_{p,f,\tf}  x^{p(\td-1/2)/d}\Bigpar{1+\sumk 2^{-k}}
=
C_{p,f,\tf}  x^{p(\td-1/2)/d}.
  \end{split}
\end{equation}

For the second term on the \rhs{} of \eqref{bb}, 
we use \eqref{bal} and \refL{Lpq}\ref{Lpqb}, with $q=\td$ and $h(x):=x^{-1/2d}$,
and conclude, for every $p\ge1$,
\begin{equation}\label{bbb}
\E \bigabs{\NN\pm(x)^{\td}-n(x)^{\td}}^p \le C_{p,f,\tf} x^{p(\td-1/2)/d}.
\end{equation}
Finally, by \refT{TRp} we have moment convergence in \eqref{tnn} and thus
\begin{equation}\label{bbd}
\E \bigpar{\NN\pm(x)^{p (\td-1)}} =O\bigpar{ x^{p(\td-1)/d}},
\end{equation}
which also follows from \eqref{Nmom} (changing $p$).

It follows from \eqref{bb} and \eqref{bbh}--\eqref{bbd}  that
\begin{equation}
\E\biggabs{
\frac{\tU_{\NN\pm(x)}-\bigparfrac{d!}{\mu}^{\td/d}\frac{\tmu}{\td!} x^{\td/d}}
{x^{(\td-1/2)d}} 
}^p
\le C_{p,f,\tf} .
\end{equation}
Since $p$ is arbitrary, this implies 
convergence of arbitrary moments in \eqref{tvtau}
by the same standard argument as in the
proof of \refT{TRp}.

Moment convergence in \eqref{tvtau0} is a corollary.
\end{proof}

\begin{proof}[Proof of \refT{TVp1}]
\pfitemref{TVp1a}
This is a special case of \refT{TVp}.

\pfitemref{TVp1b}
Denote the \lhs{} of \eqref{cvtau} by $V(x)$, for integers $x\ge1$, and let
$p>0$. 
It follows from \ref{TVp1a} that the family $|V(x)|^p$, $x\ge1$, 
is uniformly integrable. 
This property is preserved by the conditioning,
since 
we condition on a sequence of events $\cE_x$ with
$\liminf_{\xtoo}\P(\cE_x)>0$
by the proof of \refT{TO};
hence the result follows from \refT{TO}.
\end{proof}

\section{Examples and applications}\label{Sex}

\begin{example}\label{E22}
  Let $d=2$, and let $f$ be anti-symmetric: $f(y,x)=-f(x,y)$;
this case was studied in \cite{SJ22}.
We have $\mu=0$ and $f_2(x)=\E f(X,x)=-\E f(x,X)=-f_1(x)$; 
hence $\gs_{11}=-\gs_{12}=\gs_{22}$ and \eqref{WW} implies
$W_2(t)=-W_1(t)=\gs B(t)$, where $\gs:=\norm{f_1}\ge0$ and  $B(t)$ is a
standard Brownian motion.

For $d=2$, \eqref{psi} yields $\psi_1(s,t)=t-s$ and $\psi_1(s,t)=s$.
Hence, \eqref{t1}, \eqref{Z} and integration by parts, see \eqref{crux}, 
yield
\begin{equation}\label{e22}
  \begin{split}
    \frac{U_{nt}}{n^{3/2}}\dto
  Z_t
&=\int_0^t(t-2s)\dd W_1(s)
=-tW_1(t)+2\int_0^t W_1(s)\dd s
\\
&=\gs tB(t)-2\gs\int_0^t B(s)\dd s
  \end{split}
\end{equation}
in $\Doo$,
as shown in \cite{SJ22}
(where also the degenerate case $\gs=0$ is studied further).
\end{example}

\begin{example}[Substrings]\label{Estring}
  Consider a random string $X_1\dotsm X_n$ of length $n$ from a finite
  alphabet $\cA$, with the letters $X_i$ \iid{} with some distribution
  $\P(X_i=a)=p_a$, $a\in\cA$.
Fix a \emph{pattern} $\cW=w_1\dotsm w_m$; this is an arbitrary string in
$\cA^m$, for some $m\ge1$. 
A \emph{substring} of $X_1\dotsm X_n$ is any string $X_{i_1}\dotsm X_{i_k}$
with $1\le i_1<\dots<i_k\le n$, and we let 
$N_{n}=N_{\cW}(X_1\dotsm X_n)$ be the number of substrings that have the
pattern $\cW$. Obviously, this is an  asymmetric $U$-statistic as in
\eqref{U} with $\cS=\cA$, $d=m$ and
\begin{equation}
  f(x_1,\dots,x_m):=\ett{x_1\dotsm x_m=w_1\dotsm w_m}
=\prod_{i=1}^m\ett{x_i=w_i}.
\end{equation}
\refC{C1} yields asymptotic normality of $N_{n}$ as \ntoo,
as shown by \citet{FlajoletSzV}.

For example, let $\cA:=\setoi$, let $X_i\sim\Be(\frac12)$, and let $\cW:=10$.
A simple calculation yields $f_1(x)=\frac12(x-\frac12)=-f_2(x)$, and 
$\gs_{11}=\gs_{22}=-\gs_{12}=1/16$; thus \refC{C1} yields, see
\eqref{c1var2},
\begin{equation}
\frac{  N_{n}-n/4}{\sqrt n} \dto N\Bigpar{0,\frac{1}{48}}.
\end{equation}
Furthermore,
calculations as in \refE{E22} show that 
the functional limit
\eqref{e22} holds in this case too,
with $\gs=1/4$. 
\end{example}

\begin{example}[Patterns in permutations]\label{Eperm}
Let $\pi=\pi_1\dotsm\pi_n$ be a uniformly random permutation of length $n$, 
and let the \emph{pattern}
$\gs=\gs_1\dotsm\gs_m$ be a fixed permutation of length $m$.
The \emph{number of occurences} of $\gs $ in $\pi$, denoted by
$N_{n}=N_{\gs}(\pi)$ is the number of substrings (see \refE{Estring}) of
$\pi$ that have the same relative order as $\gs$.

We can generate the random permutation $\pi$ by taking \iid{} random
variables $X_1,\dots,X_n\sim \Uoi$, and then replacing these numbers by
their ranks. Then $N_{n}$  is the $U$-statistic with $d=m$ 
given by the function
\begin{equation}
  f(x_1,\dots,x_m)=\ett{x_1\dotsm x_m \text{ have the same relative order as }
\gs_1\dotsm\gs_m}.  
\end{equation}
\refC{C1} shows that $N_{n}$ is asymptotically normal as \ntoo.
For details, including explicit variance calculations, see \cite{SJ287};
see also the earlier proof of asymptotic normality by
\citet{Bona-Normal,Bona3}.

For example, taking $\gs=21$, $N_{n}$ is the number of inversions in
$\pi$, and we obtain by simple calculations the well-known result,
see \eg{} \cite[Section X.6]{FellerI},
\begin{equation}
  \frac{N_{n}-n^2/4}{n^{3/2}}\dto N\Bigpar{0,\frac1{36}}.
\end{equation}
\end{example}

\begin{example}[Restricted permutations I]\label{EpermI}
  Fix a set $T$ of permutations, and consider only permutations $\pi$ of
  length $n$ that \emph{avoid} $T$, in the sense that there is no occurence
  of any $\tau\in T$ in $\pi$. Let $\pi$ be uniformly random from this set,
for a given $n$.

Several cases are studied in \cite{SJ333}, and some of them yield
asymmetric $U$-statistics, sometimes stopped or conditioned as in
\refT{CVtau} or \ref{TO}. 
We sketch two examples here and in the next example, and refer to
\cite{SJ333} for details and further similar examples.

A permutation $\pi$ avoids $\set{\permB}$ if and only if $\pi$ is an
increasing sequence of \emph{blocks} that all are decreasing; in other words,
\begin{equation}\label{permB}
\pi=
(\ELL_1,\dots,1,\ELL_1+\ELL_2,\dots,\ELL_1+1,
\ELL_1+\ELL_2+\ELL_3,\dots,\ELL_1+\ELL_2+1,\dots),
\end{equation}
see \cite[Proposition 12]{SS}.
Let the number of blocks be $B\ge1$ and the block lengths
$\ELL_1,\dots,\ELL_B$; thus $\ELL_i\ge1$ and $\ELL_1+\dots+\ELL_B=n$.
Then, any such sequence $\ELL_1,\dots,\ELL_B$ is possible, and it 
determines $\pi$ uniquely. 
Hence,  taking $f(L):=1$ and thus $U_n=S_n=\sum_1^n \ELL_i$,
it is easily seen that $(\ELL_1,\dots,\ELL_B)$ has the
same distribution as the first $\NN-(n)$ elements of an \iid{} sequence
$(L_k)_k$ with $L_i\sim\Ge(1/2)$, conditioned on $U_{\NN-(n)}=n$.

Let $\gs$ be a fixed permutation that avoids $\set{\permB}$, with block
lengths $\ell_1,\dots,\ell_b$.
Then the number $\Ngsn=N_\gs(\pi)$
of occurrences of $\gs$ in $\pi$ is given by a $U$-statistic,
with $d=b$,
based on the sequence of variables $L_1,\dots,L_B$ and the function
\begin{equation}\label{fB}
\tf(x_1,\dots,x_b):=\prod_{j-1}^b\binom{x_i}{\ell_i}.
\end{equation}
\refT{TO}\ref{TOc} applies and shows asymptotic normality in the form
\begin{equation}
  \frac{\Ngsn-n^b/b!}{n^{b-1/2}}
\dto N\bigpar{0,\gam^2},
\end{equation}
for some $\gam^2>0$ depending on $\gs$. 

For example, taking $\gs=21$, so $\Nxn{21}$ is the number of inversions in
$\pi$, 
$b=1$ and, by a calculation, $\gam^2=6$; hence
\begin{equation}
  \frac{\Nxn{21}-n}{n^{1/2}}
\dto N\bigpar{0,6}.
\end{equation}

We here applied the conditional result in \refT{TO}. 
Alternatively (since a geometric distribution has no memory),
we may avoid the conditioning above 
and instead truncate the last element $\ELL_B$
such that the sum becomes exactly $n$;
using a simple approximation argument, we can then 
apply the unconditional
\refT{CVtau}.
\end{example}

\begin{example}[Restricted permutations II]\label{EpermII}
Continuing \refE{EpermI}, now let $\pi$ be a uniformly random permutation of
a given length $n$ such that $\pi$ avoids \set{\permAAA}.

A permutation $\pi$ avoids \set{\permAAA} if and only if $\pi$ is of the
form \eqref{permB} and furthermore every block length $L_i\le2$,
see \cite[Proposition $15^*$]{SS}.
 Taking again $f(L):=1$,
it is easily seen that $(\ELL_1,\dots,\ELL_B)$ has the
same distribution as the first $\NN-(n)$ elements of an \iid{} sequence
$(L'_k)\xoo$, conditioned on $U_{\NN-(n)}=n$, where we now let
\begin{equation}
\P(L'_i=1)=p,\qquad
\P(L'_i=2)=p^2,
\end{equation}
where $p+p^2=1$ and thus $p$ is the golden ratio
\begin{equation}
  p:=
\frac{\sqrt5-1}2.
\end{equation}

Let $\gs$ be a fixed permutation that avoids $\set{\permAAA}$, with block
lengths $\ell_1,\dots,\ell_b\in\set{1,2}$.
Then the number $\Ngsn=N_\gs(\pi)$
of occurrences of $\gs$ in $\pi$ is given by a $U$-statistic based on
$L_1,\dots,L_B$, with $d=b$ and the function $\tf$ in \eqref{fB}.
\refT{TO}\ref{TOc} applies and shows asymptotic normality in the form
\begin{equation}
  \frac{\Ngsn-\mu n^b/b!}{n^{b-1/2}}
\dto N\bigpar{0,\gam^2},
\end{equation}
for some $\mu>0$ and $\gam^2>0$ depending on $\gs$. 

For example, taking $\gs=21$, so $\Nxn{21}$ is the number of inversions in $\pi$,
$b=1$ and, by calculations, see \cite{SJ333},
$\mu=(3-\sqrt5)/2$ and
$\gam^2= 5^{-3/2}$;
hence
\begin{equation}\label{tazi}
  \frac{\Nxn{21}-\frac{3-\sqrt5}{2} n}{n^{1/2}}
\dto N\bigpar{0,5^{-3/2}}.
\end{equation}
\end{example}

\section{Further comments and open problems}\label{Sadd}

\begin{remark}\label{Rmom}
In \refTs{TRp} and \ref{TVp}, we assume (for simplicity) existence of all
moments for $f$ and $\tf$, and conclude convergence of all moments in 
\eqref{tnn}--\eqref{tvtau}. 
If we only want to conclude convergence of a specific moment, 
\eg{} convergence of second
moments in \eqref{tr} or \eqref{tvtau}, the proofs above show that it
suffices to assume existence  of some specific moment for $f$ and $\tf$. 
However, we do not know the best possible moment conditions for this, and we
leave it as an open problem to find optimal conditions. 
(The proofs above are not optimized; furthermore, the methods used there are not
necessarily optimal.)
In particular, we do not know whether  convergence of first and second moments
always holds in \eqref{tr} and \eqref{tvtau} without further moment
assumptions.
(For some results when $d=\td=1$, see \cite{SJ52} and 
\cite[Chapter 3]{Gut-SRW}.)
\end{remark}

\begin{remark}
  \label{Rlarge}
In the case when $f$ is bounded, 
subgaussian estimates for large deviations of the \lhs{} of \eqref{c1} 
are shown in \cite{Hoeffding1963} and \cite{SJ150}.
This and the definitions \eqref{NN-}--\eqref{NN+} 
lead to large deviation estimates for $\NN\pm$,
and, provided also $\tf$ is bounded, then further to large deviation
estimates for the \lhs{} in \eqref{tvtau}.
We leave the details to the reader.
\end{remark}

\begin{remark}\label{Rdeg}
As said in the introduction, 
the results above are of most interest in  the non-degenerate case, where
  $\gS=(\gs_{ij})$ defined by \eqref{gsij} is non-zero. In the
  degenerate case, when all $\gs_{ij}=0$, or equivalently, $f_i(X)=0$ \as{}
  for every $i$, the results still hold but then the
  limits in \eg{} \refT{T1} are degenerate, see also \eqref{c1=0}.
A typical degenerate example is the anti-symmetric
$f(X_1,X_2)=\sin(X_1-X_2)$, with 
$X$ uniformly distributed on $[0,2\pi)$ (best regarded as the unit
circle), where $f_1=f_2=0$.

In the degenerate case, 
one can instead normalize using a smaller power of
  $n$ than in \refT{T1} and obtain non-degenerate limits;
this is well-known   in the symmetric case, 
see \eg{} \cite{Gregory1977},  
\cite{RubinVitale1980}, 
\cite[Chapter 11]{SJIII} 
for univariate results
and 
\cite{Neuhaus1977}, 
\cite{Hall1979}, 
\cite{DehlingDP1984}, 
\cite{DenkerGK1985}, 
\cite{Ronzhin1985},  
\cite[Remark 11.11]{SJIII} 
for functional limits.
This extends to the asymmetric case;
univariate results are given in  \cite[Chapter 11.2]{SJIII} 
with the possibility of functional limits
briefly mentioned in \cite[Remark 11.25]{SJIII},  
and the case $d=2$ and $f$ antisymmetric was studied in \cite{SJ22} 
(functional limits for
both the degenerate and non-degenerate cases), see \refE{E22}.
We do not consider such refined results for the degenerate case in the
present paper.
\end{remark}

\begin{remark}
For multi-sample $U$-statistics, \ie, variables of the form
\begin{equation}
U_{n_1,\dots,n_\ell}:=
  \sum f\bigpar{X\xx1_{i_{1,1}},\dots,X\xx1_{i_{1,d(1)}},
\dots,
X\xx\ell_{i_{\ell,1}},\dots,X\xx\ell_{i_{\ell,d(\ell)}}},
\end{equation}
summing over $1\le i_{j,1}<\dots<i_{j,d(j)}\le n_j$ for every $j=1,\dots,\ell$,
a multi-dimensional functional limit theorem has been
given by \citet{Sen1974} 
in the symmetric case (\ie, with $f$ symmetric in
each of the $\ell$ sets of variables); see also \eg{}
\cite{Neuhaus1977}, 
\cite{Hall1979}, 
\cite{DenkerGK1985}. 
We expect that this too can be extended to the asymmetric case, but we leave
this to the interested reader.
\end{remark}

\begin{remark}\label{Rstand}
There is a standard trick to convert an asymmetric $U$-statistic to a symmetric
  one, see \eg{} \cite{SJIII}.
Let $Y_i\sim U(0,1)$ be \iid{} random variables, independent of
$(X_j)_1^\infty$,
let $Z_i:=(X_i,Y_i)\in \tS:=S\times\bbR$, and define $F:\tS^n\to\bbR$ by
\begin{equation}
F\bigpar{(x_1,y_1),\dots,(x_d,y_d)}:=f(x_1,\dots,x_d)\ett{y_1<\dots<y_d}
\end{equation}
and its symmetrized version 
\begin{equation}
  F^*(z_1,\dots,z_d):=\sum_{\gs\in S_d}F\bigpar{z_{\gs(1)},\dots,z_{\gs(d)}},
\end{equation}
summing over the $d!$ permutations of $\set{1,\dots,d}$.
Then, letting $\sumx$ denote the sum over distinct indices,
\begin{align}
  U_n(f)&\eqd  \sumx_{\substack{i_1,\dots,i_d\le n\\Y_{i_1}<\dots<Y_{i_d}}} 
f\bigpar{X_{i_1},\dots,X_{i_d}}
=\sumx_{i_1,\dots,i_d\le n}F\bigpar{(X_{i_1},Y_{i_1}),\dots,(X_{i_d},Y_{i_d})}
\notag\\&
=\sum_{1\le i_1<\dots<i_d\le n}F^*\bigpar{Z_{i_1},\dots,Z_{i_d}}
=U_n(F^*).
\end{align}
This trick often makes it possible to transfer results for symmetric
$U$-statistics to the general, asymmetric case. However, this trick works
only for a single $n$, and we do not know of any similar trick that can
handle the process $(U_n)_{n=0}^\infty$. Hence this method does not seem
useful for the results above.
\end{remark}

\begin{remark}\label{Rreverse}
In the symmetric case, it is easily seen that $U_n/\binom nd$, $n\ge d$, is
a reverse martingale, which for example yields a simple proof of the law of
large numbers; see \cite{Berk} and \eg{} \cite[Chapter 10.16.2]{Gut}. 
This does not hold in general; thus we used 
above
(in the proof of \refL{LU*})
instead forward martingales similarly to \cite{HoeffdingLLN}.
\end{remark}


\newcommand\AAP{\emph{Adv. Appl. Probab.} }
\newcommand\JAP{\emph{J. Appl. Probab.} }
\newcommand\JAMS{\emph{J. \AMS} }
\newcommand\MAMS{\emph{Memoirs \AMS} }
\newcommand\PAMS{\emph{Proc. \AMS} }
\newcommand\TAMS{\emph{Trans. \AMS} }
\newcommand\AnnMS{\emph{Ann. Math. Statist.} }
\newcommand\AnnPr{\emph{Ann. Probab.} }
\newcommand\CPC{\emph{Combin. Probab. Comput.} }
\newcommand\JMAA{\emph{J. Math. Anal. Appl.} }
\newcommand\RSA{\emph{Random Struct. Alg.} }
\newcommand\ZW{\emph{Z. Wahrsch. Verw. Gebiete} }
\newcommand\DMTCS{\jour{Discr. Math. Theor. Comput. Sci.} }

\newcommand\AMS{Amer. Math. Soc.}
\newcommand\Springer{Springer-Verlag}
\newcommand\Wiley{Wiley}

\newcommand\vol{\textbf}
\newcommand\jour{\emph}
\newcommand\book{\emph}
\newcommand\inbook{\emph}
\def\no#1#2,{\unskip#2, no. #1,} 
\newcommand\toappear{\unskip, to appear}

\newcommand\arxiv[1]{\texttt{arXiv:#1}}
\newcommand\arXiv{\arxiv}

\def\nobibitem#1\par{}

\newcommand\xand{\& }


\begin{thebibliography}{99}

\bibitem{Berk}
Robert H. Berk,
Limiting behavior of posterior distributions when the model is incorrect.
\emph{Ann. Math. Statist.} \textbf{37} (1966), 51--58.

\bibitem[Billingsley(1968)]{Billingsley}
Patrick Billingsley,
\emph{Convergence of Probability Measures}.
Wiley, New York, 1968.

\bibitem[B\'ona(2007)]{Bona-Normal}
Mikl\'os B\'ona,
The copies of any permutation pattern are asymptotically normal.
Preprint, 2007.
\arXiv{0712.2792}


\bibitem[B\'ona(2010)]{Bona3}
Mikl\'os B\'ona,
On three different notions of monotone subsequences.  
\emph{Permutation Patterns}, 89--114,  
London Math. Soc. Lecture Note Ser., 376, 
Cambridge Univ. Press, Cambridge, 2010.

\bibitem{DehlingDP1984}
Herold Dehling,  Manfred Denker \xand Walter Philipp,
Invariance principles for von Mises and $U$-statistics.
\emph{Z. Wahrsch. Verw. Gebiete} \textbf{67} (1984), no. 2, 139--167. 

\bibitem{DenkerGK1985}
M. Denker, Ch. Grillenberger \xand  G.  Keller, 
A note on invariance principles for v. Mises' statistics.
\emph{Metrika} \textbf{32} (1985), no. 3-4, 197--214.

\bibitem{FellerI}
William Feller, 
\emph{An Introduction to Probability Theory and Its Application}, volume I, 
third edition, Wiley, 
New York, 1968.


\bibitem[Flajolet, Szpankowski and  Vall\'ee(2006)]{FlajoletSzV}
Philippe Flajolet,
Wojciech Szpankowski
and 
Brigitte Vall\'ee.
Hidden word statistics. 
\emph{J. ACM} \textbf{53} (2006), no. 1, 147--183. 

\bibitem{Gregory1977}
Gavin G. Gregory,
 Large sample theory for $U$-statistics and tests of fit.
\emph{Ann. Statist.}
\textbf{5}   (1977), no. 1, 110--123. 

\bibitem{Gut-SRW}
Allan Gut,
\emph{Stopped Random Walks}
2nd ed. Springer, New York, 2009. 

\bibitem{Gut}
Allan Gut,
\emph{Probability: A Graduate Course},
2nd ed. Springer, New York, 2013. 

\bibitem{SJ50}
Allan Gut \xand Svante Janson,
Converse results for existence of moments and uniform integrability
for stopped random walks. 
\emph{Ann. Probab.} \textbf{14} (1986), 1296--1317.

\bibitem[Hall(1979)]{Hall1979}
Peter Hall,
On the invariance principle for $U$-statistics.
\emph{Stochastic Process. Appl.} \textbf9 (1979), no. 2, 163--174.

\bibitem[Hoeffding(1948)]{Hoeffding}
 Wassily Hoeffding, A class of statistics with asymptotically normal
 distribution. 
\emph{Ann. Math. Statistics} \textbf{19} (1948), 293--325.

\bibitem[Hoeffding(1961)]{HoeffdingLLN}
Wassily Hoeffding,
The strong law of large numbers for   $U$-statistics.
Institute of Statistics, Univ. of North Carolina, Mimeograph
      series 302 (1961).
\url{https://repository.lib.ncsu.edu/handle/1840.4/2128}


\bibitem{Hoeffding1963}
Wassily Hoeffding,
Probability inequalities for sums of bounded random variables. 
\emph{ J. Amer. Statist. Assoc.}
\textbf{58} (1963), 13--30. 


\bibitem{SJ52}
Svante Janson,
Moments for first passage and last exit times, the minimum, and
related quantities for random walks with positive drift.
\emph{Adv. Appl. Probab.}
\textbf{18} (1986), 865--879.

\bibitem{SJIII}
Svante Janson,
\emph{Gaussian Hilbert Spaces},
Cambridge Univ. Press, Cambridge, UK, 1997.

\bibitem{SJ150}
Svante Janson, 
Large deviations for sums of partly dependent random variables. 
\emph{Random Structures Algorithms} \textbf{24} (2004), no. 3, 234--248. 

\bibitem{SJ333}
Svante Janson,
Patterns in random permutations avoiding some sets of multiple patterns.
Preprint, 2018.


\bibitem[Janson, Nakamura and Zeilberger(2015)]{SJ287}
Svante Janson, Brian Nakamura \xand Doron Zeilberger,
On the asymptotic statistics of the number of occurrences of multiple
permutation patterns.
\emph{Journal of Combinatorics} \textbf{6} (2015), no. 1-2, 117--143.


\bibitem{SJ22}
Svante Janson \xand Michael J. Wichura, 
Invariance principles for stochastic area and related stochastic integrals.
\emph{Stochastic Process. Appl.} \textbf{16} (1984), no. 1, 71--84.

\bibitem{Kallenberg}
Olav Kallenberg,
\book{Foundations of Modern Probability.}
2nd ed., Springer, New York, 2002. 


\bibitem[Miller and Sen(1972)]{MillerSen}
R. G. Miller, Jr. \xand Pranab Kumar Sen,
Weak convergence of $U$-statistics and von Mises' differentiable statistical
functions.
\emph{Ann. Math. Statist.} \vol{43} (1972), 31--41.

\bibitem[Neuhaus(1977)]{Neuhaus1977}
Georg Neuhaus,
Functional limit theorems for $U$-statistics in the degenerate case.
\emph{J. Multivariate Anal.} \textbf7 (1977), no. 3, 424--439. 

\bibitem[Ronzhin(1985)]{Ronzhin1985} 
A. F. Ronzhin, 
A functional limit theorem for homogeneous $U$-statistics with degenerate
kernel. (Russian) 
\emph{Teor. Veroyatnost. i Primenen.} \textbf{30} (1985), no. 4, 759--762.
English transl.: 
\emph{Theory Probab. Appl.} 30 (1985), no. 4, 806--809. 

\bibitem{RubinVitale1980}
H. Rubin \xand R.A. Vitale, 
Asymptotic distribution of symmetric statistics.
\emph{Ann. Statist.}
\textbf8 (1980),
165--170.

\bibitem[Sen(1974)]{Sen1974}
Pranab Kumar Sen,
Weak convergence of generalized $U$-statistics
\emph{Ann. Probability} \textbf2 (1974),  90--102.


\bibitem[Simion and Schmidt(1985)]{SS}
Rodica Simion and  Frank W. Schmidt,
Restricted permutations.
\emph{European J. Combin.} \vol6 (1985), no. 4, 383--406.



\bibitem{Sproule1974} 
Raymond N. Sproule,
Asymptotic properties of $U$-statistics.
\emph{Trans. Amer. Math. Soc.} \textbf{199} (1974), 55--64. 





\end{thebibliography}
\end{document}